\documentclass[11pt,reqno]{amsart}
\usepackage{amsmath}

\usepackage{amsfonts}
\usepackage{mathrsfs}

\usepackage{amssymb}
\usepackage[dvips]{graphicx}
\usepackage{color}
\usepackage[pagewise]{lineno}
\usepackage{url}
\theoremstyle{plain}
\newtheorem{athm}{Theorem}

\newtheorem{theorem}{Theorem}[section]
\newtheorem{proposition}[theorem]{Proposition}
\newtheorem{corollary}[theorem]{Corollary}
\newtheorem{lemma}[theorem]{Lemma}

\theoremstyle{definition}

\newtheorem{example}{Example}[section]
\newtheorem{definition}[theorem]{Definition}

\newtheorem{remark}[theorem]{Remark}

\DeclareMathOperator{\esssup}{ess \ sup}

\DeclareMathOperator{\diam}{diam}

\DeclareMathOperator{\ho}{H}
\DeclareMathOperator{\Ker}{Ker}
\DeclareMathOperator{\im}{Im}
\DeclareMathOperator{\e}{e}
\DeclareMathOperator{\id}{id}
\DeclareMathOperator{\inte}{int}
\input{tcilatex}

\begin{document}

\begin{abstract}
In this article, we develop a functional-analytic framework to establish existence, uniqueness, regularity of disintegration, and statistical properties of equilibrium states for a broad class of dynamical systems, potentially discontinuous and not necessarily locally invertible. Our approach is applied to a family of piecewise partially hyperbolic maps and associated classes of potentials. We further prove several statistical limit theorems, including exponential decay of correlations, and propose related questions and conjectures. A collection of examples illustrating the applicability of our results is provided, including partially hyperbolic attractors over horseshoes, discontinuous systems, non-invertible dynamical systems admitting a semi-conjugacy to intermittent maps such as the Manneville--Pomeau map, and fat solenoidal attractors.
\end{abstract}

\title[Thermodynamic formalism for discontinuous maps]{Thermodynamic formalism for discontinuous maps and statistical properties of their equilibrium states}

\author[Rafael A. Bilbao]{Rafael A. Bilbao}
\author[Rafael Lucena]{Rafael Lucena}

\date{\today }

\address[Rafael A. Bilbao]{Escuela de Matem\'atica y Estad\'istica, Universidad Pedag\'ogica y Tecnol\'ogica de Colombia, Avenida Central del Norte 39-115, Sede Central Tunja, Boyac\'a, 150003, Colombia.}
\email{rafael.alvarez@uptc.edu.co}


\address[Rafael Lucena]{Universidade Federal de Alagoas, Instituto de Matemática - UFAL, Av. Lourival Melo Mota, S/N
	Tabuleiro dos Martins, Maceio - AL, 57072-900, Brasil}
\email{rafael.lucena@im.ufal.br}
\urladdr{www.im.ufal.br/professor/rafaellucena}
\maketitle


\section{Introduction}

Thermodynamic formalism was introduced into dynamical systems in the mid--1970s through the seminal works of Sinai, Ruelle, and Bowen \cite{Si,Ru,Bo}, establishing a deep connection between statistical mechanics and the statistical properties of chaotic dynamics. A central tool in this theory is the Ruelle--Perron--Frobenius (RPF) operator, whose spectral properties encode fundamental quantities such as entropy, pressure, equilibrium states, and decay of correlations. For a potential $\varphi$ and a continuous observable $\psi$, the RPF operator associated with a map $f$ is defined by
\[
\mathcal{L}_\varphi \psi (x) = \sum_{y \in f^{-1}(x)} \psi(y)\, e^{\varphi(y)}.
\]

Spectral methods have since become a cornerstone of thermodynamic formalism and have been successfully extended far beyond uniformly hyperbolic systems. In particular, functional--analytic approaches developed by Keller and Baladi \cite{b1, b2, b3,bt, k1, k2} enabled the study of piecewise expanding and smooth systems via bounded variation and anisotropic Banach spaces. Further major advances were achieved by Liverani, Gou\"ezel, and Tsujii \cite{g3, g1, g2,gl,  l1, l2, l3, l4, t1, t2}, who developed refined anisotropic techniques to obtain spectral gap results and sharp statistical properties for hyperbolic diffeomorphisms and flows. Related spectral approaches have also proved effective in systems with singularities and open dynamics, notably in the work of Demers, Liverani, and Zhang \cite{d5,d1,d2,d3,d4}.

In parallel, thermodynamic formalism for partially hyperbolic systems has been developed through alternative methods, particularly for systems that are semi--conjugate to non--uniformly expanding maps. This includes skew--product dynamics with contracting fibres, as studied for instance in \cite{VR1, TC, DHRS,LOR,RA1,RA2}. In these works, the base dynamics typically exhibits strong expansion properties, while the fibre dynamics is uniformly contracting, allowing the construction and analysis of invariant and equilibrium measures under various assumptions.

The present paper contributes to this line of research by developing a spectral approach to thermodynamic formalism for a class of partially hyperbolic skew products that may exhibit discontinuities, lack of invertibility, and absence of a preserved order. Our motivation is to extend the applicability of operator--theoretic methods to settings where classical assumptions do not hold, while retaining access to strong statistical properties such as existence and uniqueness of equilibrium states and exponential decay of correlations.

More precisely, we consider maps $F : M \times K \to M \times K$ of the form $F=(f,G)$, where $f: M \to M$ is a non--uniformly expanding map and $G : M \times K \to K$ is a (possibly discontinuous) fibre map exhibiting contraction along fibres. Inspired by the RPF framework and by the approaches developed in \cite{RRR, RRR2, GLu, LiLu, r1, r2}, we introduce a class of linear operators acting on suitable spaces of signed measures constructed via disintegration along the fibre partition. We show that these operators exhibit a spectral gap under general assumptions, which yields existence, uniqueness, and sharp statistical properties of equilibrium states for a class of H\"older continuous potentials.

The paper is organized as follows. In the first part, developed in Section~\ref{kjdfkjdsfkj}, we establish a functional--analytic framework adapted to the study of thermodynamic formalism in discontinuous and partially hyperbolic settings. In the second part, presented in Section~\ref{appppl}, we apply this framework to several classes of skew--product systems, including examples not previously covered by existing spectral approaches, and derive detailed statistical properties of their equilibrium states.

\noindent\textbf{Role of Theorems A--G and H--K: an informal description.} Theorems~A--G do not concern the existence of equilibrium states themselves.
Instead, they establish an abstract functional--analytic framework for
transfer operators associated with a general class of dynamical systems.
These results are conceived as tools to be applied to piecewise partially
hyperbolic skew products with contracting fibres.

More precisely, Theorems~A--G introduce suitable spaces of signed measures,
prove Lasota--Yorke type inequalities, and establish spectral properties
that constitute the backbone of the thermodynamic formalism developed in this work.
The existence, uniqueness, and statistical properties of equilibrium
states are derived later as applications of this framework;
see Theorems~H--K.

\noindent\textbf{Statements of the Main Results.} Here we present the main results of this work. For clarity, they are organised into two parts, which are presented in Subsections~\ref{ueuoiort} and~\ref{ueuoiorty}.

In what follows $F: M \times K \longrightarrow M \times K$ is a map of the type $F=(f,G)$, where $f: M \longrightarrow M$ and $G: M \times K \longrightarrow K$ are measurable maps. We suppose that $m$ and $m_2$ are fixed probabilities, where the hypothesis on the measurable space $(M, m)$ are stated in section \ref{hfgjfdh} and $m_2$ is any fixed probability on the Borel sets of the compact metric space $K$. Sometimes we use the notation $\Sigma:=M \times K$.

\subsection{General results} \label{ueuoiort}

This subsection presents the most general results of the paper, which are proved in Section~\ref{kjdfkjdsfkj}. They concern skew-product maps $F=(f,G)$ under suitable spectral assumptions: the normalized Ruelle--Perron--Frobenius operator of the base map $f$, together with the transfer operators of the fibre maps $G(x,\cdot)$, satisfy the spectral properties required in our framework. These results are formulated as Theorems~A--G.

It is worth noting that, from the perspective of continuity, this initial framework can accommodate a wide range of discontinuous systems, since the hypotheses focus primarily on the spectral properties of the operators associated with $f$ and $G$. More precisely, the assumptions are divided according to the base and fibre dynamics. 

For the base map $f$, conditions (P1) and (P2) encode the existence of equilibrium states for a suitable class of potentials, together with appropriate spectral properties of the associated Ruelle--Perron--Frobenius operator. In addition, they impose a mild structural assumption on the inverse branches of $f$. These hypotheses are standard in non-uniformly expanding settings, although the spectral decomposition required in (P1) is of a functional-analytic nature.

For the fibre map $G$, conditions (G1)--(G3) control the action of the induced fibre dynamics on spaces of signed measures, expressing uniform contraction and regularity properties. While (G1) reflects, in particular, a natural dynamical requirement for contracting fibres, conditions (G2) and (G3) are mainly technical assumptions related to the choice of norms and function spaces.

Throughout this subsection, $\mu_0$ denotes a distinguished $F$-invariant probability measure whose existence is postulated axiomatically in condition (P3). Moreover, $\mu_0$ projects onto the equilibrium state $m$ of $f$, a hypothesis that is commonly satisfied in skew-product constructions. At this level of generality, $\mu_0$ is not assumed a priori to possess additional properties, such as being an equilibrium state for a given potential. Its precise dynamical and statistical nature is progressively clarified throughout this section and further developed in Section~\ref{kjdfkjdsfkj}, where we establish uniqueness, spectral characterizations, and statistical properties.

The following two results concern the action of the operators 
$\func{\overline{F}}_\Phi$ and $\func{\overline{F}}_{\Phi,h}$ 
associated with potentials $\Phi$ belonging to the class considered in this framework 
(see Equation (\ref{PPP}) and axiom (P1)). Although this convergence property is weaker than the existence of a spectral gap, 
it will be a key ingredient in the derivation of such stronger spectral properties in subsequent results.
They establish exponential rates of convergence of the iterates 
$(\func{\overline{F}}_\Phi^{n}\mu)_{n\geq 1}$ and 
$(\func{\overline{F}}_{\Phi,h}^{n}\mu)_{n\geq 1}$ 
towards the equilibrium state $\mu_0$, 
for all probability measures $\mu$ belonging to $\mathbf{S}^\infty$ and $\mathbf{S}^1$, respectively, 
with respect to the weak norms associated with the spaces 
$\mathbf{L}^{1}$ and $\mathbf{L}^{\infty}$.

The normed spaces $\mathbf{L}^{1}$, $\mathbf{L}^{\infty}$, 
$\mathbf{S}^{1}$, and $\mathbf{S}^{\infty}$ are defined in 
Definitions~\ref{l1}, \ref{linffff}, \ref{s1111}, and~\ref{sinf}, respectively. 
The sets $\mathbf{V}^{1}$ and $\mathbf{V}^{\infty}$ denote the linear subspaces of 
zero-average measures of $\mathbf{S}^{1}$ and $\mathbf{S}^{\infty}$, respectively; 
see Equations~\eqref{mathV} and~\eqref{mathVV} for the precise definitions.

\begin{athm}[Exponential convergence to the equilibrium in $\mathbf{L}^{\infty}$ and $\mathbf{S}^{\infty}$]\label{5.8}
Suppose that $F$ satisfies (P1), (P2), (P3), (G1), Equation~(\ref{Olga1}) of (G2), and (G3).
Then there exist constants $D \in \mathbb{R}$ and $0 \leq \beta_3 < 1$ such that, 
for every signed measure $\mu \in \mathbf{V}^\infty$, it holds
\begin{equation*}
	\|\func{\overline{F}}_{\Phi,h}^{\,n}\mu\|_{\infty}
	\leq D\,\beta_3^{\,n}\,\|\mu\|_{\mathbf{S}^{\infty}},
\end{equation*}
for all $n \geq 1$.
\end{athm}

\begin{athm}[Exponential convergence to the equilibrium in $\mathbf{L}^1$ and $\mathbf{S}^{1}$]\label{5.9}
Suppose that $F$ satisfies (P1), (P2), (P3), (G1), Equation~(\ref{Olga2}) of (G2), and (G3).
Then there exist constants $D_4 \in \mathbb{R}$ and $0 \leq \beta_4 < 1$ such that, 
for every signed measure $\mu \in \mathbf{V}^1$, it holds
\begin{equation*}
	\|\func{\overline{F}}_\Phi^{\,n}\mu\|_{1}
	\leq D_4\,\beta_4^{\,n}\,\|\mu\|_{\mathbf{S}^{1}},
\end{equation*}
for all $n \geq 1$.
\end{athm}

The next result gives the necessary hypothesis to guaranties uniqueness (since the existence is given by (P3)) of an $F$-invariant measure in $\mathbf{S}^{\infty}$ and in $\mathbf{S}^{1}$ (see Equations (\ref{s1}) and (\ref{sinfi})). Its proof is given in section \ref{invt}.

\begin{athm}\label{belongss}
Suppose that $F$ satisfies (P1), (P2), (P3), (G1), Equation~(\ref{Olga1}) of (G2), and (G3). Then $F$ admits a unique $F$-invariant probability measure in $\mathbf{S}^{\infty}$. If instead $F$ satisfies (P1), (P2), (P3), (G1), Equation~(\ref{Olga2}) of (G2), and (G3),
then $F$ admits a unique $F$-invariant probability measure in $\mathbf{S}^{1}$. In both cases, the invariant measure $\mu_0$ satisfies $\pi_{1*}\mu_0 = m$,
where $\pi_1 : \Sigma \to M$ denotes the canonical projection onto $M$ and
$\pi_{1*}$ its associated pushforward map.
\end{athm}

The following Theorems~\ref{spgap} and~\ref{spgapp}, proved in Section~\ref{jshdjfgsjhdf}, show that the operator $\func{\overline{F}}_\Phi$ acting on the spaces $\mathbf{S}^{1}$ and $\mathbf{S}^{\infty}$ has a spectral gap. Results of this type have far-reaching dynamical and statistical consequences, as they imply several limit and mixing properties. In particular, as established in Theorems~\ref{shkjfjdhsf} and~\ref{slkdgjsdg}, we obtain exponential decay of correlations, namely
\[
\lim_{n \to \infty} C_n(u_1,u_2)=0,
\]
where
\[
C_n(u_1,u_2):=\left| \int (u_2 \circ F^n)\, u_1 \, d\mu_0 - \int u_2 \, d\mu_0 \int u_1 \, d\mu_0 \right|,
\]
for $u_1 \in B'_1$ and $u_2 \in \Theta_{\mu_0}^1$ in Theorem~\ref{shkjfjdhsf}, and for $u_1 \in B'_\infty$ and $u_2 \in \Theta_{\mu_0}^\infty$ in Theorem~\ref{slkdgjsdg}.

\begin{athm}[Spectral gap on $\mathbf{S}^{1}$]
\label{spgap}
Suppose that $F$ satisfies (P1), (P2), (P3), (G1), Equation~(\ref{Olga2}) of (G2), and (G3).
Then the operator $\func{\overline{F}}_\Phi:\mathbf{S}^{1}\longrightarrow \mathbf{S}^{1}$
admits a decomposition
\[
\func{\overline{F}}_\Phi=\func{P}+\func{N},
\]
where:
\begin{enumerate}
	\item[(a)] $\func{P}$ is a projection of rank one, that is, $\func{P}^2=\func{P}$ and $\dim \im(\func{P})=1$;
	
	\item[(b)] there exist constants $0<\xi<1$ and $R>0$ such that, for all $\mu\in\mathbf{S}^{1}$,
	\[
	\|\func{N}^n(\mu)\|_{\mathbf{S}^{1}}\le R\,\xi^n\|\mu\|_{\mathbf{S}^{1}}, \quad \forall n\ge1;
	\]
	
	\item[(c)] $\func{P}\func{N}=\func{N}\func{P}=0$.
\end{enumerate}
\end{athm}

\begin{athm}[Spectral gap on $\mathbf{S}^{\infty}$]
\label{spgapp}
Suppose that $F$ satisfies (P1), (P2), (P3), (G1), Equation~(\ref{Olga1}) of (G2), and (G3).
Then the operator $\func{\overline{F}}_{\Phi,h}:\mathbf{S}^{\infty}\longrightarrow \mathbf{S}^{\infty}$
admits a decomposition
\[
\func{\overline{F}}_{\Phi,h}=\func{P}+\func{N},
\]
where:
\begin{enumerate}
	\item[(a)] $\func{P}$ is a projection of rank one, that is, $\func{P}^2=\func{P}$ and $\dim \im(\func{P})=1$;
	
	\item[(b)] there exist constants $0<\xi _2<1$ and $R_2>0$ such that, for all $\mu\in\mathbf{S}^{\infty}$,
	\[
	\|\func{N}^n(\mu)\|_{\mathbf{S}^{\infty}}\le R_2\,\xi_2 ^n\|\mu\|_{\mathbf{S}^{\infty}}, \quad \forall n\ge1;
	\]
	
	\item[(c)] $\func{P}\func{N}=\func{N}\func{P}=0$.
\end{enumerate}
\end{athm}

Before stating Theorems~\ref{shkjfjdhsf} and~\ref{slkdgjsdg}, we introduce the classes of observables that will be used in their formulation and proofs. Let $\mathbf{S}^1$ and $\mathbf{S}^\infty$ denote the corresponding spaces of signed measures introduced above. We define
\begin{equation}\label{sjdfhjdsf}
	B'_1 := \left\{ u : \Sigma \to \mathbb{R} \;\middle|\;
	\exists\, R_3 = R_3(u);
	\left| \int u \, d\mu \right| \le R_3 \, \|\mu\|_{\mathbf{S}^1},
	\ \forall \mu \in \mathbf{S}^1 \right\},
\end{equation}
and
\begin{equation}\label{poyuoypuiyu}
	B'_\infty := \left\{ u : \Sigma \to \mathbb{R} \;\middle|\;
	\exists\, R_4 = R_4(u);
	\left| \int u \, d\mu \right| \le R_4 \, \|\mu\|_{\mathbf{S}^\infty},
	\ \forall \mu \in \mathbf{S}^\infty \right\}.
\end{equation}
Moreover, given the reference invariant measure $\mu_0$, we set
\begin{equation*}
	\Theta_{\mu_0}^1 := \{ v : \Sigma \to \mathbb{R} : v \mu_0 \in \mathbf{S}^1 \}
	\quad \text{and} \quad
	\Theta_{\mu_0}^\infty := \{ v : \Sigma \to \mathbb{R} : v \mu_0 \in \mathbf{S}^\infty \}.
\end{equation*}
Here, for a measurable function $v$ and a measure $\mu$, the signed measure
$v\mu$ is defined by $(v\mu)(A) := \int_A v \, d\mu$ for every measurable
set $A \subset \Sigma$.

\begin{athm}\label{shkjfjdhsf}
Suppose that $F$ satisfies (P1), (P2), (P3), (G1), Equation~(\ref{Olga1}) of (G2), and (G3). 
Then, for all observables $u \in B'_1$ and $v \in \Theta_{\mu_0}^1$, it holds that
\[
\left| \int (u \circ F^n) v  d\mu_0 
      - \int u  d\mu_0 \int v  d\mu_0 \right|
\leq \| v \mu_0 \|_{\mathbf{S}^{1}} \, R R_3  \xi^{n},
\qquad \forall n \geq 1,
\]
where $\xi$ and $R$ are as in Theorem~\ref{spgap}, and $R_3=R_3(u)$ is as in
Equation~(\ref{sjdfhjdsf}).
\end{athm}

\begin{athm}\label{slkdgjsdg}
Suppose that $F$ satisfies (P1), (P2), (P3), (G1), Equation~(\ref{Olga2}) of (G2), and (G3). 
Then, for all observables $u \in B'_\infty$ and $v \in \Theta_{\mu_0}^\infty$, it holds that
\[
\left| \int (u \circ F^n) v d\mu_0 
      - \int u d\mu_0 \int v d\mu_0 \right|
\leq \|v \mu_0 \|_{\mathbf{S}^{\infty}} R R_4 \xi^{n},
\qquad \forall n \geq 1,
\]
where $\xi$ and $R$ are as in Theorem~\ref{spgapp}, and $R_4=R_4(u)$ is as in
Equation~(\ref{poyuoypuiyu}).
\end{athm}

\subsection{Applications}\label{ueuoiorty}

The following theorems are applications of the results established above and will be proved in Section~\ref{appppl}. They concern piecewise partially hyperbolic skew-products of the form $F=(f,G)$, where $f$ is a mostly expanding local homeomorphism satisfying conditions (f1), (f2), and (f3), and $G$ is a (possibly discontinuous) map satisfying conditions (H1) and (H2), all stated precisely in Section~\ref{appppl}. 

Assumptions (f1)--(f3) ensure that $f$ is a non-uniformly expanding transformation whose associated Ruelle--Perron--Frobenius operator admits a spectral gap for suitable H\"older potentials. On the other hand, assumptions (H1) and (H2) concern the fibre map $G$ and guarantee uniform contraction along almost every fibre, together with a mild regularity condition allowing for discontinuities.

In this setting, we show that $F$ admits an invariant measure $\mu_0$ satisfying condition~(P3). Moreover, in Theorem~\ref{belongsss} we prove that, if $F$ uniformly contracts all vertical fibres, then $\mu_0$ is an equilibrium state for $F$ associated with a lifted potential $\Phi \in \mathscr{P}_\Sigma$, where
\begin{equation}\label{ps1}
\mathscr{P}_\Sigma
=
\left\{
\Phi:\Sigma \to \mathbb{R}
\;\middle|\;
\exists\, \varphi \in \mathscr{P}_M \text{ such that } \Phi=\varphi \circ \pi_1
\right\},
\end{equation}
$\pi_1:\Sigma \to M$ denotes the projection onto the first coordinate, and $\mathscr{P}_M$ is a class of potentials $\varphi:M \to \mathbb{R}$ (see Definitions~\ref{PH} and~\ref{P}).

\begin{athm}\label{belongsss}
Assume that $F$ satisfies conditions (f1), (f2), (f3), and (H1). 
Then, for each potential $\Phi \in \mathscr{P}_\Sigma$, there exists a unique $F$-invariant measure $\mu_0 \in \mathbf{S}^1$, which in addition belongs to $\mathbf{S}^\infty$. In particular, if $F$ uniformly contracts all fibres, then $\mu_0$ is an equilibrium state.
\end{athm}

We now state a theorem that establishes a quantitative regularity property of the unique equilibrium state $\mu_0$ associated with $(F,\Phi)$. More precisely, it provides an estimate for the 
Hölder constant (see Equation~\eqref{Lips2} in Definition~\ref{Lips3}) of the disintegration of $\mu_0$.

Results of this type have several applications. Similar regularity estimates were obtained for other classes 
of dynamical systems in \cite{RRR2,GLu}, where they are used to prove stability of the $F$-invariant measure 
under suitable perturbations. In the present setting, we apply this theorem to show that the abstract sets 
$\Theta^1_{\mu_0}$, $\Theta^\infty_{\mu_0}$, $B'_1$, and $B'_\infty$, introduced in 
Theorems~\ref{shkjfjdhsf} and~\ref{slkdgjsdg}, contain the space of $\zeta$-Hölder functions 
$\ho_\zeta(\Sigma)$.

\begin{athm}\label{regg}
Suppose that $F:\Sigma \to \Sigma$ satisfies conditions (f1), (f2), (f3), (H1), (H2), and $(\alpha \cdot L)^\zeta < 1$.
For each potential $\Phi \in \mathscr{P}_\Sigma$, let $\mu_0$ denote the $F$-invariant measure of Theorem \ref{belongsss}.
Then $\mu_0 \in \mathcal{H}_\zeta^{+}$ and
\begin{equation}\label{VA}
|\mu_0|_\zeta \leq \frac{D}{1-\beta},
\end{equation}
where $D$ and $\beta$ are the constants given in Proposition~\ref{iuaswdas}.
\end{athm}

As a consequence of the estimate established in the previous theorem, the next result complements Theorems~\ref{shkjfjdhsf} and~\ref{slkdgjsdg}. In particular, it implies that the limit $$\lim_{n \to \infty} C_n(u,v) = 0$$ converges exponentially fast for all Hölder observables $u, v \in \ho_\zeta(\Sigma)$. The proof is given in Subsection~\ref{jhdfjhdf}.

\begin{athm}\label{disisisii}
Suppose that $F:\Sigma \longrightarrow \Sigma$ satisfies conditions
(f1), (f2), (f3), (H1), (H2), and $(\alpha \cdot L)^\zeta<1$.
Let $\mu_0$ denote the unique $F$-invariant measure associated
with a potential $\Phi \in \mathscr{P}_\Sigma$, which belongs to both
$\mathbf{S}^1$ and $\mathbf{S}^\infty$.
Then
\[
\ho_\zeta(\Sigma) \subset \Theta^1_{\mu_0}, \qquad
\ho_\zeta(\Sigma) \subset \Theta^\infty_{\mu_0}, \qquad
\ho_\zeta(\Sigma) \subset B'_1, \qquad
\ho_\zeta(\Sigma) \subset B'_\infty,
\] where $\ho_\zeta(\Sigma)$ denotes the space of real $\zeta$-H\"olders functions defined on $\Sigma$. 
\end{athm}

A natural question is whether the results obtained so far extend to a broader class of potentials than $\mathscr{P}_\Sigma$ (see Equation (\ref{ps1})), in particular to potentials that are not necessarily constant along the fibres. To address this question, and to isolate it as a co-homological problem, we introduce the class $\mathcal{S}$ of dynamical systems $F$ satisfying the following properties:
\begin{enumerate}
	\item $F$ satisfies assumptions (f1), (f2), and (f3), and uniformly contracts all vertical fibres;
	\item there exists a class of potentials $\overline{\mathscr{P}}_\Sigma$ such that, for every
	$\overline{\Phi} \in \overline{\mathscr{P}}_\Sigma$, one can associate a potential
	$\Phi \in \mathscr{P}_\Sigma$ for which the pairs $(F,\Phi)$ and $(F,\overline{\Phi})$
	share the same equilibrium states.
\end{enumerate}Note that, in order to belong to the class $\mathcal{S}$, the map $G$ must satisfy a condition stronger than~(H1). An application of this framework is provided in Theorem~\ref{S}.

This viewpoint shows that the restriction to fibre-constant potentials is not a limitation of the method, but rather a natural choice of representative within a cohomology class. Extending the theory to non-constant fibre potentials is therefore a co-homological problem, rather than an operator-theoretic one.

The next theorem provides a concrete realization of the class $\mathcal{S}$,
together with an associated class of fibre potentials $\overline{\mathscr{P}}_\Sigma$
that are not necessarily constant along the fibres, for which all the results
established in this work apply.
The proof relies on the previous results, together with co-homological arguments
from \cite{RA3} (see Proposition~\ref{homologo}).
An example of partially hyperbolic attractors over horseshoes satisfying these
assumptions is presented in Example~\ref{ferradura}.

\begin{athm}\label{S}
Suppose that the continuous dynamics $F:\Sigma \longrightarrow \Sigma$
satisfies conditions (f1), (f2), (f3), and uniformly contracts all vertical fibres. Moreover, there exists
$y_0 \in K$ such that
\begin{equation}\label{iurytea}
G(x,y_0)=y_0, \quad \forall\, x \in M.
\end{equation}
Then every H\"older continuous potential
$\bar{\Phi} : M \times K \to \mathbb{R}$ such that $\bar{\varphi}_{y_0} \in \mathscr{P}_M$,
where $\bar{\varphi}_{y_0} : M \to \mathbb{R}$ is defined by
\[
\bar{\varphi}_{y_0}(x) := \bar{\Phi}(x,y_0), \quad x \in M,
\]
admits a unique equilibrium state $\mu_0 \in \mathbf{S}^1$, which moreover
belongs to $\mathbf{S}^\infty$.
Furthermore, if $F$ also satisfies (H2) and $(\alpha \cdot L)^\zeta < 1$,
then $\mu_0$ satisfies Theorems~\ref{regg} and~\ref{disisisii}. That is,
\begin{equation*}
|\mu_0|_\zeta \leq \frac{D}{1-\beta},
\end{equation*}
and
\[
\ho_\zeta(\Sigma) \subset \Theta^1_{\mu_0}, \qquad
\ho_\zeta(\Sigma) \subset \Theta^\infty_{\mu_0}, \qquad
\ho_\zeta(\Sigma) \subset B'_1, \qquad
\ho_\zeta(\Sigma) \subset B'_\infty.
\]
\end{athm}

More generally, we believe that additional examples of systems belonging to
$\mathcal{S}$ may be obtained through co-homological arguments that allow one to
replace a given potential by a co-homologous one that is constant along the fibres.
This type of phenomenon has already appeared in different contexts. For instance,
in Proposition~7.1 of Ferreira and Ramos~\cite{FR}, it is shown that H\"older
continuous potentials satisfying the Bowen property are co-homologous to potentials
that do not depend on the stable direction, while preserving the relevant
thermodynamic properties.

Although the class of potentials considered in~\cite{FR} is different from the
class $\overline{\mathscr{P}}_\Sigma$ introduced here, this result provides further
evidence that the problem of extending thermodynamic results to non--fibre--constant
potentials can often be formulated as a co-homological one.

\noindent\textbf{Open questions.} From this point onward, we raise the following questions:
\begin{enumerate}
	\item Is this functional-analytic framework suitable for proving statistical stability results, as in \cite{RRR2}? Moreover, could it serve as an effective approach to establish other limit theorems, such as the Central Limit Theorem and exponential decay of correlations, in alternative spaces of observables?

	\item Is this new linear operator also suitable for proving linear and quadratic (even higher-order) responses for these maps?

	\item Is it also applicable to systems that are not skew-products?
	
	\item Regarding the applications section (see Section \ref{appppl}), can this approach be applied to study non-contracting fibre maps?

    \item Can co-homological arguments be systematically employed to extend the class
$\mathcal{S}$, by identifying conditions under which non--fibre--constant
potentials are co-homologous to fibre--constant ones, while preserving the relevant
thermodynamic properties?
\end{enumerate}

\noindent\textbf{Plan of the paper.} The paper is structured as follows:

\begin{itemize}
	\item Section 2: we present the general results and settings of this article. We consider a class of systems of the form $F(x,y)=(f(x), G(x,y))$, where the Ruelle--Perron--Frobenius operator of $f$ has a spectral gap, and the transfer operator associated with $G(x,\cdot)$ satisfies certain properties with respect to the $||\cdot||_o$-norm on the vector spaces of signed measures on $K$, for $m$-almost every $x \in M$. These properties are precisely described by Equations~(\ref{gjhfjghfjgj}), (\ref{weakkkk}), and (\ref{Olga1}) or (\ref{Olga2}). In this section, we prove Theorems~\ref{5.8}–\ref{slkdgjsdg}.
	
	\item Section 3: we apply the results of Section 2 to maps $F=(f,G)$, where $f$ is a non-uniformly expanding map and $G$ is a contracting fibre map, possibly with vertical lines of discontinuities. In this section, we prove Theorems~\ref{belongsss}–\ref{S}.
\end{itemize}

\textbf{Acknowledgments.} We are deeply grateful to Ricardo Liberalquino Bioni for his invaluable contributions to this group. We also thank Marcelo Viana and Paulo Varandas for their important insights. Finally, we extend special thanks to Vanessa Ramos and Martin Anderson for their thoughtful comments and support.

\section{General Results}\label{kjdfkjdsfkj}

Throughout this section, we consider skew--product maps
$F : M \times K \longrightarrow M \times K$ of the form $F=(f,G)$, where
$f : M \longrightarrow M$ and $G : M \times K \longrightarrow K$
are measurable maps specified below.
We fix a probability measure $m$ on $M$ satisfying the hypotheses listed in (P1),
and an arbitrary probability measure $m_2$ on the Borel $\sigma$--algebra of the
compact metric space $K$.
For convenience, we set $\Sigma := M \times K$.

\subsubsection{Hypotheses on the base map $f$}

\begin{enumerate}

\item[(P1)] We assume the following conditions.

\begin{enumerate}

\item[(P1.1)] There exists a class of potentials $\mathscr{P}_M$ such that, for each
$\varphi \in \mathscr{P}_M$, there is a corresponding equilibrium state
$m = m(\varphi)$ of $f$ which can be written as $m = h \, \nu$, where
\[
0 < \inf h \leq \sup h < +\infty,
\]
and
\[
\mathcal{L}_\varphi h = \lambda h \quad \text{for some } \lambda > 0.
\]
Moreover, $\nu$ is a conformal measure with Jacobian $\lambda e^{-\varphi}$.

\item[(P1.2)] There exist normed spaces of real-valued functions (not necessarily Banach),
closed under pointwise multiplication,
\[
h, h^{-1} \in (B_s, |\cdot|_s) \subset (B_w, |\cdot|_w),
\qquad |\cdot|_w \leq |\cdot|_s,
\]
such that for every $\varphi \in \mathscr{P}_M$ the Ruelle--Perron--Frobenius operator
$\mathcal{L}_\varphi : B_s \to B_s$, defined by
\begin{equation}\label{hhdfhjdghf}
\mathcal{L}_\varphi g(x)
= \sum_{y \in f^{-1}(x)} g(y) e^{\varphi(y)},
\qquad g \in B_s,
\end{equation}
has a spectral decomposition of the form
\[
\operatorname{spec}(\mathcal{L}_\varphi)
= \{\lambda\} \cup \mathcal{U},
\]
where $\mathcal{U}$ is contained in a disk of radius strictly smaller than $|\lambda|$.

In particular, the normalized operator
\[
\overline{\mathcal{L}}_\varphi := \frac{1}{\lambda}\mathcal{L}_\varphi
\]
admits a decomposition
\[
\overline{\mathcal{L}}_\varphi = \mathcal{P}_f + \mathcal{N}_f,
\]
where
\begin{enumerate}
\item $\mathcal{P}_f^2 = \mathcal{P}_f$ and $\dim \operatorname{Im}(\mathcal{P}_f) = 1$;
\item there exist constants $D > 0$ and $0 \le r < 1$ such that
\[
|\mathcal{N}_f^n g|_s \le D r^n |g|_s, \qquad \forall g \in B_s;
\]
\item $\mathcal{N}_f \mathcal{P}_f = \mathcal{P}_f \mathcal{N}_f = 0$.
\end{enumerate}

Consequently, if $\mathcal{L}_\varphi^*$ denotes the dual of $\mathcal{L}_\varphi$,
then $\mathcal{L}_\varphi^* \nu = \lambda \nu$, and hence
\begin{equation}\label{fixeddd}
\overline{\mathcal{L}}_\varphi^* \nu = \nu.
\end{equation}

\end{enumerate}

\item[(P2)] There exists a measurable partition of $M$ into disjoint sets
$P_1, \dots, P_{\deg(f)}$ (up to $\nu$--null sets) such that, for each $i$,
the restriction $f_i := f|_{P_i}$ is a bijection onto $M$, and
\[
\nu\Big( \bigcup_i P_i \Big) = \nu(M).
\]

\end{enumerate}

\medskip

In the sequel, we assume the existence of an $F$--invariant probability measure
whose projection onto the base coincides with the $f$--invariant measure $m$.

\begin{enumerate}
\item[(P3)] There exists an $F$--invariant probability measure $\mu_0$ such that
\[
\pi_{1*}\mu_0 = m,
\]where $\pi_1 : \Sigma \to M$ denotes the projection $\pi_1(x,y)=x$ and $\pi_{1*}$
is the associated pushforward.
\end{enumerate}

By (P1), the following theorems hold for the normalized operator $$\mathcal{\overline{L}}_\varphi:=\dfrac{1}{\lambda}\mathcal{L}_\varphi.$$ 
\begin{theorem}\label{LYgeral}
	There exist constants $B_1>0$, $C_1>0$ and $0<\beta_1<1$ such that for all $%
	g \in B_s$, and all $n \geq 1$, it holds
	
	\begin{equation*}
		|\mathcal{\overline{L}}_\varphi^n(g)|_s \leq B_1 \beta _1 ^n | g|_s + C_1|g|_{w}.
		\label{lasotaiiii}
	\end{equation*}
\end{theorem}

\begin{proof}
By hypothesis (P1.2), we have the decomposition
\[
\overline{\mathcal L}_\varphi = \mathcal P_f + \mathcal N_f,
\]
where $\mathcal P_f^2 = \mathcal P_f$ and
$\mathcal N_f \mathcal P_f = \mathcal P_f \mathcal N_f = 0$.
Therefore, for all $n \ge 1$,
\[
\overline{\mathcal L}_\varphi^n = \mathcal P_f + \mathcal N_f^n.
\]
Hence,
\[
|\overline{\mathcal L}_\varphi^n g|_s
\le |\mathcal P_f g|_s + |\mathcal N_f^n g|_s
\le |\mathcal P_f g|_s + D r^n |g|_s.
\]

Since $\mathcal P_f g = \overline{\mathcal L}_\varphi^n \mathcal P_f g$ for every $n \in \mathbb N$, we obtain
\[
|\mathcal P_f g|_s
= \frac{|\mathcal P_f g|_s}{|\mathcal P_f g|_w}
\, |\overline{\mathcal L}_\varphi^n \mathcal P_f g|_w.
\]
Using the triangle inequality, this yields
\[
|\mathcal P_f g|_s
\le
\frac{|\mathcal P_f g|_s}{|\mathcal P_f g|_w}
\left(
|\overline{\mathcal L}_\varphi^n g|_w
+ |\overline{\mathcal L}_\varphi^n (g - \mathcal P_f g)|_w
\right).
\]

Since $\dim(\operatorname{Im}(\mathcal P_f)) = 1$, all norms are equivalent on
$\operatorname{Im}(\mathcal P_f)$, and therefore the ratio
$\frac{|\mathcal P_f g|_s}{|\mathcal P_f g|_w}$ is uniformly bounded.
Moreover, $\overline{\mathcal L}_\varphi$ is bounded on $B_w$, and
$|\overline{\mathcal L}_\varphi^n (g - \mathcal P_f g)|_w \le D r^n |g|_s$.
Combining these estimates concludes the proof.
\end{proof}The next theorem is a direct consequence of (P1.2), so we omit the proof.

\begin{theorem}\label{loiub}
There exist constants $0 < r < 1$ and $D > 0$ such that, for every
$g \in \Ker(\mathcal P_f)$ and every $n \ge 1$, one has
\[
|\overline{\mathcal L}_\varphi^{\,n}(g)|_s \le D r^n |g|_s.
\]
\end{theorem}

\begin{remark}\label{yturhfvb}
Recall that the Jacobian of the measure $m$ with respect to $f$ is given by
\[
J_m(f) = \lambda e^{-\varphi} \frac{h \circ f}{h}.
\]
Moreover, since $0 < \inf h \le \sup h < +\infty$ (see (P1.1)),
the measures $m$ and $\nu$ are equivalent.

Define the conjugated transfer operator
$\mathcal{L}_{\varphi,h} : B_s \to B_s$ by
\begin{equation}\label{jhfdgjhfg}
\mathcal{L}_{\varphi,h}(g)
:= \frac{\mathcal{L}_\varphi(g h)}{h},
\qquad g \in B_s.
\end{equation}
A direct computation shows that $\mathcal{L}_{\varphi,h}$ satisfies (P1.1) and (P1.2) with $h \equiv 1$.
In particular, $\mathcal{L}_{\varphi,h}(1) = \lambda.$
\end{remark}

\begin{remark}\label{jhdfgjsd}
The hypotheses of the previous theorems are also satisfied by the operator
$\mathcal{L}_{\varphi,h}$ defined in Remark~\ref{yturhfvb} by
Equation~\eqref{jhfdgjhfg}. Consequently, there exist constants
$B_3 > 0$, $C_3 > 0$ and $0 < \beta_3 < 1$ such that, for every $g \in B_s$
and every $n \ge 1$, one has
\begin{equation}\label{lasotaiiiityrd}
|\overline{\mathcal{L}}_{\varphi,h}^{\,n}(g)|_s
\le B_3 \beta_3^n |g|_s + C_3 |g|_w.
\end{equation}
Moreover, there exist constants $D_3 > 0$ and $0 < r_3 < 1$ such that, for every
$g \in \Ker(\mathcal{P}_f)$ and every $n \ge 1$, it holds
\begin{equation}\label{lasotaiisdffrr}
|\overline{\mathcal{L}}_{\varphi,h}^{\,n}(g)|_s
\le D_3 r_3^n |g|_s.
\end{equation}
\end{remark}

\subsubsection{Hypotheses on the fibre map $G$}

Let $\mathcal{SM}(K)$ denote the vector space of signed measures on $K$, endowed
with the usual operations of addition and scalar multiplication.
We define the family of fibres
\[
\mathcal{F} := \{ \gamma \subset \Sigma \mid \gamma = \pi_1^{-1}(x), \ x \in M \}.
\]
For a given $\gamma \in \mathcal{F}$, we define the induced fibre map
$F_\gamma : K \to K$ by
\begin{equation}\label{ritiruwt}
F_\gamma := \pi_2 \circ F|_{\gamma} \circ \pi_{\gamma,2}^{-1},
\end{equation}
where $\pi_2(x,y) = y$ and $\pi_{\gamma,2}$ denotes the restriction of $\pi_2$
to the fibre $\gamma$.

\begin{enumerate}
\item[(G1)] There exist a norm $\|\cdot\|_o$ on $\mathcal{SM}(K)$ and a constant
$0 \le \alpha < 1$ such that the following properties hold:
\begin{enumerate}
\item[(G1.1)] For every $\mu \in \mathcal{SM}(K)$ and for $\nu$-almost every
fibre $\gamma \in \mathcal{F}$, one has
\begin{equation}\label{gjhfjghfjgj}
\|F_{\gamma *}\mu\|_o \le \alpha \|\mu\|_o + |\mu(K)|.
\end{equation}
In particular, if $\mu(K)=0$, then for $\nu$-almost every $\gamma \in \mathcal{F}$
it holds
\[
\|F_{\gamma *}\mu\|_o \le \alpha \|\mu\|_o.
\]

\item[(G1.2)] For every $\mu \in \mathcal{SM}(K)$ and every $\gamma \in \mathcal{F}$,
\begin{equation}\label{weakkkk}
\|F_{\gamma *}\mu\|_o \le \|\mu\|_o.
\end{equation}
\end{enumerate}
\end{enumerate}

The following hypothesis will only become clear after
Definition~\ref{linffff}. We state it here since it concerns the
$\|\cdot\|_o$--norm on measures.

Unlike the previous assumptions, we do not assume that the system
satisfies both Equations~\eqref{Olga1} and~\eqref{Olga2}.
Instead, we explicitly indicate in which results each of these
conditions is required.

\begin{enumerate}
\item[(G2)] One of the following conditions holds:
\begin{equation}\label{Olga1}
B_w \subset L_\nu^\infty
\quad \text{and} \quad
|\phi|_\infty \le |\phi|_w \le \|\mu\|_\infty,
\qquad \forall \mu \in \mathbf{L}^\infty,
\end{equation}
or
\begin{equation}\label{Olga2}
B_w \subset L_\nu^1
\quad \text{and} \quad
|\phi|_1 \le |\phi|_w \le \|\mu\|_1,
\qquad \forall \mu \in \mathbf{L}^1.
\end{equation}

\item[(G3)] For every probability measure $\mu$, one has
\[
\|\mu\|_o = 1.
\]
\end{enumerate}

\subsection{The spaces and the operators}\label{hfgjfdh}

\subsubsection*{Rokhlin's Disintegration Theorem}

We briefly recall some standard results and definitions concerning the
disintegration of measures.

Let $(\Sigma,\mathcal{B},\mu)$ be a probability space and let $\Gamma$ be a
partition of $\Sigma$ into measurable sets $\gamma \in \mathcal{B}$. Denote by
$\pi : \Sigma \to \Gamma$ the projection that associates to each point
$x \in \Sigma$ the element $\gamma_x \in \Gamma$ containing $x$, that is,
$\pi(x) = \gamma_x$. Let $\widehat{\mathcal{B}}$ be the $\sigma$--algebra on
$\Gamma$ induced by $\pi$, namely, a subset $\mathcal{Q} \subset \Gamma$ is
measurable if and only if $\pi^{-1}(\mathcal{Q}) \in \mathcal{B}$. The
\emph{quotient measure} $\mu_x$ on $\Gamma$ is defined by
\[
\mu_x(\mathcal{Q}) := \mu\bigl(\pi^{-1}(\mathcal{Q})\bigr).
\]

The proof of the following theorem can be found in Theorem~5.1.11 and Proposition~5.1.7 of~\cite{Kva}.

\begin{theorem}[Rokhlin's Disintegration Theorem]\label{rok}
Suppose that $\Sigma$ is a complete and separable metric space, $\Gamma$ is a
measurable partition of $\Sigma$, and $\mu$ is a probability measure on
$\Sigma$. Then $\mu$ admits a disintegration relative to $\Gamma$, that is,
there exist a family $\{\mu_\gamma\}_{\gamma \in \Gamma}$ of probability
measures on $\Sigma$ and a quotient measure $\mu_x$ such that:
\begin{enumerate}
\item[(a)] $\mu_\gamma(\gamma) = 1$ for $\mu_x$--almost every $\gamma \in \Gamma$;

\item[(b)] for every measurable set $E \subset \Sigma$, the map
\[
\Gamma \ni \gamma \longmapsto \mu_\gamma(E)
\]
is measurable;

\item[(c)] for every measurable set $E \subset \Sigma$, one has
\[
\mu(E) = \int_\Gamma \mu_\gamma(E)\, d\mu_x(\gamma);
\]

\item[(d)] if the $\sigma$--algebra $\mathcal{B}$ admits a countable generator,
then the disintegration is unique in the following sense: if
$\{\mu'_\gamma\}_{\gamma \in \Gamma}$ is another family of probability measures
such that $(\{\mu'_\gamma\},\mu_x)$ is a disintegration of $\mu$ relative to
$\Gamma$, then
\[
\mu_\gamma = \mu'_\gamma
\quad \text{for } \mu_x\text{--almost every } \gamma \in \Gamma.
\]
\end{enumerate}
\end{theorem}

\subsubsection{The spaces}
Let $\mathcal{SB}(\Sigma )$ be the space of Borel  signed measures on $\Sigma : = M \times K$. Given $\mu \in \mathcal{SB}(\Sigma )$ denote by $\mu ^{+}$ and $\mu ^{-}$
the positive and the negative parts of its Jordan decomposition, $\mu =\mu
^{+}-\mu ^{-}$ (see remark {\ref{ghtyhh}). Let $\pi _1:\Sigma
	\longrightarrow M$ \ be the projection defined by $\pi_1 (x,y)=x$, denote
	by $\pi _{1\ast}:$}$\mathcal{SB}(\Sigma )\rightarrow \mathcal{SB}(M)${\
	the pushforward map associated to $\pi _1$. Denote by $\mathbf{AB}_\nu$ the
	set of signed measures $\mu \in \mathcal{SB}(\Sigma )$ such that its
	associated positive and negative marginal measures, $\pi _{1\ast }\mu ^{+}$
	and $\pi _{1\ast }\mu ^{-},$ are absolutely continuous with respect to $\nu$, i.e.,
	\begin{equation*}
		\mathbf{AB}_\nu=\{\mu \in \mathcal{SB}(\Sigma ):\pi _{1\ast }\mu ^{+} \ll \nu\ \ 
		\mathnormal{and}\ \ \pi _{1\ast }\mu ^{-} \ll \nu\}.  \label{thespace1}
	\end{equation*}%
}Given a \emph{probability measure} $\mu \in \mathbf{AB}_\nu$ on $\Sigma $,
Theorem \ref{rok} describes a disintegration $\left( \{\mu _{\gamma
}\}_{\gamma },\mu _{x}\right) $ along $\mathcal{F}$ by a family $\{\mu _{\gamma }\}_{\gamma }$ of probability measures
on the stable leaves and, since 
$\mu \in \mathbf{AB}_\nu$, $\mu _{x}$ can be identified with a non negative
marginal density $\phi _{1}:M\longrightarrow \mathbb{R}$, defined almost
everywhere, with $|\phi _{1}|_{1}=1$. \ For a general (non normalized)
positive measure $\mu \in \mathbf{AB}_\nu$ we can define its disintegration in
the same way. In this case, $\mu _{\gamma }$ are still probability measures, $%
\phi _{1}$ is still defined and $\ |\phi _{1}|_{1}=\mu (\Sigma )$.

Analogously, we define the space 
\begin{equation*}
	\mathbf{AB}_m=\{\mu \in \mathcal{SB}(\Sigma ):\pi _{1\ast }\mu ^{+} \ll m\ \ 
	\mathnormal{and}\ \ \pi _{1\ast }\mu ^{-} \ll m\}.  \label{thespace1sd}
\end{equation*}

\begin{remark}
Note that $\mathbf{AB}_m = \mathbf{AB}_\nu$, as the measures $m$ and $\nu$ are equivalent. The only difference lies in the fact that for a given measure $\mu$, the densities $\dfrac{d\pi _{1 \ast}\mu}{d\nu}$ and $\dfrac{d\pi _{1 \ast}\mu}{dm}$ may differ.
\end{remark}

\begin{definition}
	Let $\pi _{2}:\Sigma \longrightarrow K$ be the projection defined by $%
	\pi _{2}(x,y)=y$. Let $\gamma \in \mathcal{F}$, consider $\pi
	_{\gamma ,2}:\gamma \longrightarrow K$, the restriction of the map $\pi
	_{2}:\Sigma \longrightarrow K$ to the vertical leaf $\gamma $, and the
	associated pushforward map $\pi _{\gamma ,2\ast }$. Given a positive measure 
	$\mu \in \mathbf{AB}_i$ ($i=\nu,m$) and its disintegration along the stable leaves $%
	\mathcal{F}$, $\left( \{\mu _{\gamma }\}_{\gamma },\mu _{x}=\phi
	_{1} i \right)$ for $i=\nu,m$, we define the \textbf{restriction of $\mu $ on $\gamma $}
	and denote it by $\mu |_{\gamma }$ as the positive measure on $K$ (not
	on the leaf $\gamma $) defined, for all measurable set $A\subset K$, as 
	\begin{equation*}
		\mu |_{\gamma }(A)=\pi _{\gamma ,2\ast }(\phi _{1}(\gamma )\mu _{\gamma
		})(A).
	\end{equation*}%
	For a given signed measure $\mu \in \mathbf{AB}_i \ (i=\nu,m)$ and its Jordan
	decomposition $\mu =\mu ^{+}-\mu ^{-}$, define the \textbf{restriction of $%
		\mu $ on $\gamma $} by%
	\begin{equation*}
		\mu |_{\gamma }=\mu ^{+}|_{\gamma }-\mu ^{-}|_{\gamma }.
	\end{equation*}%
	\label{restrictionmeasure}
\end{definition}

\begin{remark}
	\label{ghtyhh}As proved in Appendix 2 of \cite {GLu}, the restriction $%
	\mu |_{\gamma }$ does not depend on the decomposition. Precisely, if $\mu
	=\mu _{1}-\mu _{2}$, where $\mu _{1}$ and $\mu _{2}$ are any positive
	measures, then $\mu |_{\gamma }=\mu _{1}|_{\gamma }-\mu _{2}|_{\gamma }$ $%
	m_{1}$-a.e. $\gamma \in M$. Moreover, as showed in \cite{GLu}, the restriction is linear in the sense that $(\mu_1 + \mu_2)|_\gamma = \mu_1|_\gamma + \mu_2|_\gamma$.
\end{remark}

\begin{definition}\label{l1}
	Let $\mathbf{L}^{1}_\nu\subseteq \mathbf{AB}_\nu(\Sigma )$ be defined as%
	\begin{equation*}
		\mathbf{L}^{1}_\nu=\left\{ \mu \in \mathbf{AB}_\nu:\int ||\mu
		|_{\gamma }||_od\nu<\infty \right\}.
	\end{equation*}%
	Define the function $||\cdot ||_{1}:\mathbf{L}^{1}_\nu\longrightarrow \mathbb{R}$ by%
	\begin{equation*}
		||\mu ||_{1}=\int ||\mu
		|_{\gamma }||_od\nu.
	\end{equation*}
\end{definition}

\begin{definition}\label{linffff}
	Let $\mathbf{L}^{\infty }_m\subseteq \mathbf{AB}_m(\Sigma )$ be defined as%
	\begin{equation*}
		\mathbf{L}^{\infty}_m=\left\{ \mu \in \mathbf{AB}_m:\esssup (||\mu
		^{+}|_{\gamma }-\mu ^{-}|_{\gamma }||_o<\infty \right\},
	\end{equation*}%
	where the essential supremum is taken over $M$ with respect to $m$.
	Define the function $||\cdot ||_{\infty }:\mathbf{L}^{\infty
	}_m\longrightarrow \mathbb{R}$ by%
	\begin{equation*}
		||\mu ||_{\infty }=\esssup ||\mu ^{+}|_{\gamma }-\mu ^{-}|_{\gamma
		}||_o.
	\end{equation*}
\end{definition}Given that the spaces $\mathbf{L}^{\infty }_m$ and $\mathbf{L}^{1}_\nu$ depend on $m$ and $\nu$ respectively, and there is no risk of confusion, we will refer to them simply as $\mathbf{L}^{\infty }$ and $\mathbf{L}^{1}$ from this point forward.

\begin{definition}\label{s1111}We define the strong space $\mathbf{S}^1$ by
	\begin{equation}\label{s1}
		\mathbf{S}^{1}=\left\{ \mu \in \mathbf{L}^{1};\phi _{1}\in B_s \right\},
	\end{equation}%
	and the function, $||\cdot ||_{\mathbf{S}^{1}}:\mathbf{S}^{1}\longrightarrow 
	\mathbb{R}$, defined by%
	\begin{equation*}
		||\mu ||_{\mathbf{S}^{1}}=|\phi _{1}|_{s}+||\mu ||_{1}.
	\end{equation*}
\end{definition}

\begin{definition}\label{sinf}
	Finally, consider the following set of signed measures on $\Sigma $%
	\begin{equation}\label{sinfi}
		\mathbf{S}^{\infty }=\left\{ \mu \in \mathbf{L}^{\infty };\phi _{1}\in
		B_s \right\},
	\end{equation}%
	and the function, $||\cdot ||_{\mathbf{S}^{\infty }}:\mathbf{S}^{\infty }\longrightarrow 
	\mathbb{R}$, defined by%
	\begin{equation*}
		||\mu ||_{\mathbf{S}^{\infty }}=|\phi _{1}|_s+||\mu ||_{\infty }.
	\end{equation*}
\end{definition}
The proof of the next Proposition is straightforward, so we skip the proof.

\begin{proposition}
	$\left( \mathbf{L}^{\infty },||\cdot ||_{\infty }\right) $, $\left( \mathbf{L}^{1},||\cdot ||_{1}\right)$, $\left(
	\mathbf{S}^{1},||\cdot||_{\mathbf{S}^{1}}\right)$ and $\left(
	\mathbf{S}^{\infty},||\cdot||_{\mathbf{S}^{\infty}}\right) $ are normed vector spaces.
\end{proposition}

Finally, we define the following sets of zero-average measures in $\mathbf{S}^{1}$ and $\mathbf{S}^{\infty}$:
\begin{equation}
\mathbf{V}^1 := \{ \mu \in \mathbf{S}^{1} : \phi_1 \in \Ker(\mathcal{P}_f) \}, \label{mathVV}
\end{equation}
and
\begin{equation}
\mathbf{V}^\infty := \{ \mu \in \mathbf{S}^\infty : \phi_1 \in \Ker(\mathcal{P}_f) \}. \label{mathV}
\end{equation}

\subsubsection{The operators}In this section, we define linear operators $\func{F} _\Phi:\mathbf{AB}_\nu \longrightarrow \mathbf{AB}_\nu$ and $\func{F} _{\Phi,h}:\mathbf{AB}_m \longrightarrow \mathbf{AB}_m$ by selecting appropriate expressions for the disintegration of the measures $\func{F} _\Phi \mu$ and $\func{F} _{\Phi,h} \mu$. To achieve this, we also consider the transfer operator of $F$, defined by the standard expression $\func F_\ast \mu (A) = \mu(F ^{-1}(A))$ for any measure $\mu$ and any measurable set $A$.

In what follows, the potentials (defined on $\Sigma$) considered belong to the class 

\begin{equation}\label{PPP}  
\mathscr{P}_\Sigma = \{ \Phi: \Phi := \varphi \circ \pi _1, \varphi \in \mathscr{P}_M \}. 
\end{equation}

From now on, $\chi _{A}$ stands for the characteristic function of $A$.

\begin{definition}\label{fphi}
Let $\Phi$ be a potential defined on $\Sigma$. We define the linear operator
\[
\func{F}_\Phi : \mathbf{AB}_\nu \longrightarrow \mathbf{AB}_\nu
\]
as follows. For every $\mu \in \mathbf{AB}_\nu$, the marginal on the base is given by
\begin{equation*}
\label{1}
(\func{F}_\Phi \mu)_x := \mathcal{L}_\varphi(\phi_1)\,\nu.
\end{equation*}
For $\gamma \in M$ such that $\mathcal{L}_\varphi(\phi_1)(\gamma) \neq 0$, and for each inverse branch
$\gamma_i \in f^{-1}(\gamma)$ associated with the partition element $P_i$, define the weight
\[
w_i(\gamma) := \phi_1(\gamma_i)\, e^{\varphi(\gamma_i)}.
\]
Then the conditional measure of $\func{F}_\Phi \mu$ on the fibre over $\gamma$ is defined by
\begin{equation*}
\label{2}
(\func{F}_\Phi \mu)_\gamma
=
\frac{1}{\mathcal{L}_\varphi(\phi_1)(\gamma)}
\sum_{i=1}^{\deg(f)}
w_i(\gamma)\,
\func{F}_* \mu_{\gamma_i}\,
\chi_{f(P_i)}(\gamma).
\end{equation*}
If $\mathcal{L}_\varphi(\phi_1)(\gamma)=0$, we define $(\func{F}_\Phi \mu)_\gamma$ to be the Lebesgue measure on the fibre over $\gamma$ (any fixed reference probability measure could be chosen instead).
\end{definition}

\begin{corollary}\label{fff}
For every $\mu \in \mathbf{AB}_\nu$ and $\mu_x$-a.e. $\gamma \in M$ it holds	$$(\func{F} _\Phi \mu)|_\gamma= \sum _{i=1}^{\deg(f)}{\func {F}_{\gamma_i*}\mu|_{\gamma_i}e^{\varphi(\gamma_i)}}.$$
\end{corollary}

It is straightforward to prove that $\func{F} _\Phi$ is a well defined linear operator. Its proof is a consequence of the linearity of the restriction cited in Remark \ref{ghtyhh}. Thus, we skip it.

\begin{proposition}\label{rqtwytrrwewe}
Let $m = h \nu$ be such that $\mathcal{L}_\varphi(h)=\lambda h$. If $\mu_0$ is an $F$-invariant measure with $\pi_{1*}\mu_0 = m$, then $\mu_0$ is an eigenvector for $\func{\overline{F}}_\Phi$ with eigenvalue $\lambda$.
\end{proposition}

\begin{proof}

Consider the measurable partition ${P_1,\ldots,P_{\deg(f)}}$ of $M$ provided by (P2). Recall that $\mathcal{L}_\varphi(h)=\lambda h$ and that $\nu$ is a conformal measure with Jacobian $\lambda e^{-\varphi}$. Let $g:\Sigma \longrightarrow \mathbb{R}$ be a continuous function and set $s(\gamma):=\int _{\Sigma} g\circ F\,d\mu_{0,\gamma}$. Using the definition of $\func{\overline{F}}_\Phi$ (Definition~\ref{fphi}), the disintegration of $\mu_0$ and $m=h\nu$, we obtain
\begin{align*}
\int g\,d\func{\overline{F}}_\Phi\mu_0
&= \int_M \int_{\Sigma} g\,d(\func{\overline{F}}_\Phi\mu_0)_\gamma \,d(\func{\overline{F}}_\Phi\mu_0)_x(\gamma)\\
&= \int_M \frac{1}{\mathcal{L}_\varphi(h)(\gamma)}
   \sum_{i=1}^{\deg(f)} h(\gamma_i)e^{\varphi(\gamma_i)}
   \int _{\Sigma} g\,d(\func{F}_*\mu_{0,\gamma_i})\,d\mathcal{L}_\varphi(h)\nu(\gamma)\\
&= \int_M \sum_{i=1}^{\deg(f)} h(\gamma_i)e^{\varphi(\gamma_i)}
   \int _{\Sigma} g\circ F\,d\mu_{0,\gamma_i}\,d\nu(\gamma)\\
&= \sum_{i=1}^{\deg(f)}\int_{f(P_i)} h(\gamma_i)e^{\varphi(\gamma_i)} s(\gamma_i)\,d\nu(\gamma)\\
&= \sum_{i=1}^{\deg(f)}\int_{P_i} \mathcal{L}_\varphi(h)(\gamma) s(\gamma)\,d\nu(\gamma)
   \;=\; \lambda\int_M h(\gamma) s(\gamma)\,d\nu(\gamma)\\
&= \lambda\int_M \int _{\Sigma} g\circ F\,d\mu_{0,\gamma}\,dm(\gamma)
   \;=\; \lambda\int_\Sigma g\circ F\,d\mu_0
   \;=\; \lambda\int_\Sigma g\,d\mu_0.
\end{align*}
This shows $\int _{\Sigma} g\,d\func{\overline{F}}_\Phi\mu_0=\int _{\Sigma} g\,d(\lambda\mu_0)$ for every continuous function $g$, hence $\func{\overline{F}}_\Phi\mu_0=\lambda\mu_0$, as claimed.
\end{proof}

The reciprocal is also true. We skip the proof, since it is straightforward.

\begin{proposition}\label{khgjgh}
	If $\mu_0 \in \mathbf{AB}_\nu$ is an eigenvector of $\func {F}_\Phi$ with eigenvalue $\lambda$, then $\mathcal{L}_\varphi (\phi_1)=\lambda \phi_1$.
\end{proposition}

By Propositions \ref{rqtwytrrwewe} and \ref{khgjgh} we get the following result.
\begin{corollary}
	The measure $\mu_0 \in \mathbf{AB}_\nu$ is a maximal eigenvector for $\func {F_\Phi}$ if and only if $\phi_1$ is a maximal eigenvector for $\mathcal{L}_\varphi$. 
\end{corollary}

Now, we define a new operator, $\func {F_{\Phi,h}}$, associated with $h$ as follows.

\begin{definition}\label{fphih}
Let $\Phi$ be a potential defined on $\Sigma$. We define the linear operator
\[
\func{F}_{\Phi,h} : \mathbf{AB}_m \longrightarrow \mathbf{AB}_m
\]
as follows. For every $\mu \in \mathbf{AB}_m$, the marginal on the base is given by
\begin{equation*}
\label{1}
(\func{F}_{\Phi,h} \mu)_x := \mathcal{L}_{\varphi,h}(\phi_1)\, m.
\end{equation*}
For $\gamma \in M$ such that $\mathcal{L}_{\varphi,h}(\phi_1)(\gamma) \neq 0$, and for each inverse branch
$\gamma_i \in f^{-1}(\gamma)$ associated with the partition element $P_i$, define the weight
\[
w_i^{(h)}(\gamma) := \phi_1(\gamma_i)\, h(\gamma_i)\, e^{\varphi(\gamma_i)}.
\]
Then the conditional measure of $\func{F}_{\Phi,h} \mu$ on the fibre over $\gamma$ is defined by
\begin{equation*}
\label{2}
(\func{F}_{\Phi,h} \mu)_\gamma
=
\frac{1}{h(\gamma)\, \mathcal{L}_{\varphi,h}(\phi_1)(\gamma)}
\sum_{i=1}^{\deg(f)}
w_i^{(h)}(\gamma)\,
\func{F}_* \mu_{\gamma_i}\,
\chi_{f(P_i)}(\gamma).
\end{equation*}
If $\mathcal{L}_{\varphi,h}(\phi_1)(\gamma)=0$, we define $(\func{F}_{\Phi,h} \mu)_\gamma$ to be the Lebesgue measure on the fibre over $\gamma$ (any fixed reference probability measure could be chosen instead).
\end{definition}

\begin{corollary}\label{fff}
	For every $\mu \in \mathbf{AB}_m$ and $\mu_x$-a.e. $\gamma \in M$, it holds	$$(\func{F} _{\Phi,h} \mu)|_\gamma= \frac{1}{h(\gamma)}\sum _{i=1}^{\deg(f)}{\func {F}_{\gamma_i*}\mu|_{\gamma_i}h(\gamma_i)e^{\varphi(\gamma_i)}}.$$
\end{corollary}

It is straightforward to prove that $\func{F} _{\Phi,h}$ is a well defined linear operator. Its proof is a consequence of the linearity of the restriction cited in Remark \ref{ghtyhh}. Thus, we skip it.

\subsubsection{Properties of the norms, actions of $\func {\overline{F}}_\Phi$ and $\func {\overline{F}}_{\Phi,h}$}

\begin{proposition}\label{uuytu}
The operator $\func{F}_\Phi:\mathbf{L}^{1}\longrightarrow \mathbf{L}^{1}$ is bounded.
Moreover, there exists $\lambda>0$ such that
\[
\|\func{F}_\Phi \mu\|_1 \leq \lambda \|\mu\|_1,
\quad \forall\, \mu \in \mathbf{L}^{1}.
\]
\end{proposition}

\begin{proof}
Consider the partition of $M$ into measurable and disjoint sets $P_1, \dots, P_{\deg(f)}$ given by (P2). 
Recall that $m = h \, \nu$, where $\mathcal{L}_\varphi(h) = \lambda h$, and that the Jacobian of $\nu$ is given by $\lambda e^{-\varphi}$. 
By Corollary \ref{fff}, Equation (\ref{weakkkk}), and the triangle inequality, we have
\begin{align*}
\int_M \| (\func{F}_\Phi \mu)|_\gamma \|_o \, d\nu 
&\le \sum_{i=1}^{\deg(f)} \int_{f(P_i)} e^{\varphi(\gamma_i)} \, \| \func{F}_{\gamma_i *} \mu |_{\gamma_i} \|_o \, d\nu \\
&\le \sum_{i=1}^{\deg(f)} \int_{f(P_i)} e^{\varphi(\gamma_i)} \, \| \mu |_{\gamma_i} \|_o \, d\nu \\
&\le \lambda \sum_{i=1}^{\deg(f)} \int_{P_i} \| \mu |_{\gamma_i} \|_o \, d\nu \\
&= \lambda \int_M \| \mu |_\gamma \|_o \, d\nu \\
&= \lambda \| \mu \|_1.
\end{align*}
This shows that $\func{F}_\Phi: \mathbf{L}^{1} \longrightarrow \mathbf{L}^{1}$ is bounded, with operator norm at most $\lambda$.
\end{proof}

\begin{definition}\label{normalized}
Define the operator $\func {\overline{F}}_{\Phi,h}$, by $$\func {\overline{F}}_{\Phi,h}:=\dfrac{1}{\lambda} \func {F}_{\Phi,h}.$$
\end{definition}

By Proposition \ref{uuytu} we immediately have the following result.

\begin{corollary}\label{urhirkjdfhkdf}
The normalized operator $\func{\overline{F}}_{\Phi,h}:\mathbf{L}^{1}\longrightarrow \mathbf{L}^{1}$ is a weak contraction, that is,
\[
\|\func{\overline{F}}_{\Phi,h}\mu\|_1 \leq \|\mu\|_1,
\quad \forall\, \mu \in \mathbf{L}^{1}.
\]
\end{corollary}

The proof of the following proposition is straightforward, so we omit it. 
\begin{corollary}\label{disintnorm}
	For every $\mu \in \mathbf{AB}_m$, the restriction of the measure $\func {\overline{F}}_{\Phi,h} \mu$ to the leaf $\gamma$, $(\func {\overline{F}}_{\Phi,h} \mu )|_\gamma$, is given by the expression
	$$ (\func {\overline{F}}_{\Phi,h} \mu )|_\gamma = \dfrac{1}{\lambda} (\func{F} _{\Phi,h} \mu)|_\gamma.$$	
\end{corollary}

\begin{proposition}\label{kdjfjksdkfhjsdfk}
	The operator $\func {\overline{F}}_{\Phi,h}: \mathbf{L}^{\infty} \longrightarrow \mathbf{L}^{\infty}$ is a weak contraction. It holds, $||\func {\overline{F}}_{\Phi,h} \mu ||_\infty \leq ||\mu||_\infty$, for all $\mu \in \mathbf{L}^{\infty}$.
\end{proposition}

\begin{proof}
Applying Equation (\ref{weakkkk}), Corollary \ref{fff}, and Remark \ref{yturhfvb}, we obtain
\begin{align*}
\| (\func{\overline{F}}_{\Phi,h} \mu)|_\gamma \|_o 
&\le \frac{1}{h(\gamma)} \sum_{i=1}^{\deg(f)} \left\| \frac{1}{\lambda} \func{F}_{\gamma_i *} \mu|_{\gamma_i} \, h(\gamma_i) e^{\varphi(\gamma_i)} \right\|_o \\
&\le \frac{1}{h(\gamma)\, \lambda} \sum_{i=1}^{\deg(f)} h(\gamma_i) e^{\varphi(\gamma_i)} \, \| \mu|_{\gamma_i} \|_o \\
&\le \frac{1}{h(\gamma)\, \lambda} \, \| \mu \|_\infty \sum_{i=1}^{\deg(f)} h(\gamma_i) e^{\varphi(\gamma_i)} \\
&= \frac{\| \mu \|_\infty}{\lambda} \, \mathcal{L}_{\varphi,h}(1)(\gamma).
\end{align*}
Taking the essential supremum over $\gamma \in M$ with respect to $m$, we conclude
\[
\| \func{\overline{F}}_{\Phi,h} \mu \|_\infty \le \| \mu \|_\infty,
\]
which proves that $\func{\overline{F}}_{\Phi,h}$ is a weak contraction.
\end{proof}

The next lemma describes the disintegration of measures absolutely continuous with respect to $\mu_0$ in terms of the disintegration of $\mu_0$. Its proof can be found in \cite{RRR}, Lemma 8.1.

\begin{lemma}\label{hdgfghddsfg}
Let $(\{\mu_{0,\gamma}\}_\gamma, h)$ be the disintegration of $\mu_0$ along the partition 
\[
\mathcal{F} := \{\{\gamma\} \times K : \gamma \in M\}.
\] 
Let $s : \Sigma \to \mathbb{R}$ be $\mu_0$-integrable and define the measure $\kappa := s\, \mu_0$, that is,
\[
\kappa(E) := \int_E s \, d\mu_0, \quad \text{for all measurable } E \subset \Sigma.
\]
Let $(\{\kappa_\gamma\}_\gamma, \kappa _x)$ be the disintegration of $\kappa$, where $\kappa_x := \pi_{1*} \kappa$. Then:
\begin{enumerate}
    \item $\kappa_x \ll \nu$ and $\kappa_\gamma \ll \mu_{0,\gamma}$;
    \item Denoting
    \[
        \overline{s} := \frac{d\kappa _x}{d\nu},
    \]
    we have
    \begin{equation}\label{fjgh}
        \overline{s}(\gamma) = \int_K s(\gamma, y) \, d(\mu_0|_\gamma)(y), \quad \text{for $\nu$-a.e. } \gamma \in M;
    \end{equation}
    \item For $\kappa _x$-a.e. $\gamma \in M$, the Radon-Nikodym derivative of $\kappa_\gamma$ with respect to $\mu_{0,\gamma}$ satisfies
    \begin{equation*}
        \frac{d\kappa_\gamma}{d\mu_{0,\gamma}}(y) =
        \begin{cases}
            \dfrac{s|_\gamma(y)}{\int_K s|_\gamma(y) \, d\mu_{0,\gamma}(y)}, & \gamma \in B^c,\\[1mm]
            0, & \gamma \in B,
        \end{cases} \quad \text{for all } y \in K,
    \end{equation*}
    where $B := \overline{s}^{-1}(0)$.
\end{enumerate}
\end{lemma}

The following Propositions \ref{olgaa1} and \ref{olgaa11111} are fundamental, as they provide a precise characterization of the operators $\func{\overline{F}}_\Phi$ and $\func{\overline{F}}_{\Phi,h}$ (see Remark~\ref{who} below).

\begin{proposition}\label{olgaa1}
For all $\mu_0 \in \mathbf{AB}_\nu$ and all $\mu_0$-integrable functions $g$ and $s$ it holds 
\begin{equation*}
\int{(g \circ F)\cdot s}d\mu_0 = \int{g}d\func {\overline{F}}_\Phi (s\mu_0).
\end{equation*}	
\end{proposition}
\begin{proof}
To simplify the notation, define
\[
(s\mu_0)_x(\gamma):=\frac{d\,\pi_{1*}(s\mu_0)}{d\nu}(\gamma),
\qquad
c(\gamma):=(s\mu_0)_x(\gamma)\int g\circ F\,d(s\mu_0)_\gamma,
\]
where $(\{(s\mu_0)_\gamma\}_\gamma,(s\mu_0)_x)$ is the disintegration of $s\mu_0$
given by Lemma~\ref{hdgfghddsfg}.

By~\eqref{fixeddd} and Definition~\ref{fphi},
\begin{align*}
\int g\,d\func{\overline{F}}_\Phi(s\mu_0)
&= \int_M \int_\Sigma g \, d(\func{\overline{F}}_\Phi(s\mu_0))_\gamma
   \, d\mathcal{\overline{L}}_\varphi (s\mu_0)_x \nu \\
&= \int_M \frac{1}{\lambda}
   \sum_{y\in f^{-1}(\gamma)} (s\mu_0)_x(y)e^{\varphi(y)}
   \int_\Sigma g\,d(\func{F}_*(s\mu_0)_y)\, d\nu \\
&= \int_M \frac{1}{\lambda}
   \sum_{y\in f^{-1}(\gamma)} (s\mu_0)_x(y)e^{\varphi(y)}
   \int_\Sigma g\circ F\,d(s\mu_0)_y \, d\nu \\
&= \int_M \mathcal{\overline{L}}_\varphi(c)\, d\nu
 = \int_M c \, d\mathcal{\overline{L}}_\varphi^*\nu
 = \int_M c\, d\nu \\
&= \int_M (s\mu_0)_x(\gamma)\int g\circ F\, d(s\mu_0)_\gamma \, d\nu
 = \int_\Sigma g\circ F \, d(s\mu_0) \\
&= \int_\Sigma (g\circ F)\, s \, d\mu_0,
\end{align*}
as claimed.
\end{proof}

Analogously, we have the next proposition. The proof is omitted.

\begin{proposition}\label{olgaa11111}
	For all $\mu_0 \in \mathbf{AB}_m$ and all $\mu_0$-integrable functions $g$ and $s$, it holds that
	\begin{equation*}
		\int (g \circ F) \cdot s \, d\mu_0 = \int g \, d\func{\overline{F}}_{\Phi,h} (s\mu_0).
	\end{equation*}	
\end{proposition}

\begin{remark}\label{who}
	Since
	\[
	\int (g \circ F) \cdot s \, d\mu_0 = \int g \, d\func{F}_*(s\mu_0),
	\]
	where $\func{F}_*$ is the classical transfer operator of $F$, the propositions above imply
	\[
	\func{F}_*(s\mu_0) = \func{\overline{F}}_\Phi(s\mu_0) \quad \text{and} \quad 
	\func{F}_*(s\mu_0) = \func{\overline{F}}_{\Phi,h}(s\mu_0)
	\]
	for all signed measures $s\mu_0$ that are absolutely continuous with respect to $\mu_0$, with $\mu_0 \in \mathbf{AB}_i$, $i=\nu,m$, respectively.
\end{remark}

\begin{proposition}[Lasota--Yorke inequality on $\mathbf{S}^1$]
\label{lasotaoscilation2}
Suppose that $F$ satisfies (P1), (P2), (P3), (G1), Equation~(\ref{Olga2}) of (G2), and (G3).  
Then there exist constants $A>0$, $B_{2}>0$, and $0<\beta _2<1$ such that, for all $\mu \in \mathbf{S}^{1}$ and all $n\ge 1$, it holds
\begin{equation*}
\|\func{\overline{F}}_\Phi^{n}\mu\|_{\mathbf{S}^1}
\leq A \beta _2 ^{n} \|\mu\|_{\mathbf{S}^{1}} + B_{2} \|\mu\|_{1}.
\end{equation*}
\end{proposition}

\begin{proof}
Firstly, recall that $\phi_1$ is the marginal density of the disintegration of $\mu$. Precisely,
\[
\phi_1 = \phi_1^+ - \phi_1^-, \quad \text{with} \quad 
\phi_1^+ = \frac{d\pi_{1*}\mu^+}{d\nu}, \quad \phi_1^- = \frac{d\pi_{1*}\mu^-}{d\nu}.
\]By Equation~(\ref{Olga2}) of (G2), Theorem~\ref{LYgeral}, Definition~\ref{s1111}, and Corollary~\ref{urhirkjdfhkdf}, we obtain
\begin{align*}
\|\func{\overline{F}}_\Phi^n \mu \|_{\mathbf{S}^1} 
&= |\mathcal{\overline{L}}_\varphi^n \phi_1|_s + \|\func{\overline{F}}_\Phi^n \mu\|_1 \\
&\le B_1 \beta_1^n |\phi_1|_s + C_1 |\phi_1|_w + \|\mu\|_1 \\
&\le B_1 \beta_1^n \|\mu\|_{\mathbf{S}^1} + (C_1 + 1) \|\mu\|_1.
\end{align*}The claim follows by setting $\beta_2 = \beta_1$, $A = B_1$, and $B_2 = C_1 + 1$.
\end{proof}

\begin{proposition}[Lasota--Yorke inequality for $\mathbf{S}^\infty$] 
\label{lasotaoscilation22}
For all $\mu \in \mathbf{S}^\infty$ and all $n \ge 1$, it holds
\begin{equation*}
\|\func{\overline{F}}_{\Phi,h}^{\,n}\mu\|_{\mathbf{S}^\infty} 
\le A \beta_2^n \|\mu\|_{\mathbf{S}^\infty} + B_2 \|\mu\|_\infty.
\end{equation*}
\end{proposition}

\begin{proof}
The technical part of the proof is analogous to Proposition~\ref{lasotaoscilation2}, and we omit the details.  
The main differences are: we use Equation~(\ref{Olga1}) instead of (\ref{Olga2}), Proposition~\ref{kdjfjksdkfhjsdfk} instead of Corollary~\ref{urhirkjdfhkdf}, and Equation~(\ref{lasotaiiiityrd}) in place of the corresponding bound for $\mathbf{S}^1$.
\end{proof}

\subsection{Convergence to the equilibrium}\label{invt}

In this section, we establish two key results regarding exponential convergence to the equilibrium. The first, Theorem \ref{5.8}, addresses the norms in $\mathbf{L}^\infty$ and $\mathbf{S}^\infty$. The second, Theorem \ref{5.9}, concerns the normed spaces $\mathbf{L}^1$ and $\mathbf{S}^1$.

We categorize the results presented here into two groups, each dedicated to proving one of the aforementioned theorems. These groups are determined by the Equation (\ref{Olga1}) or (\ref{Olga2}) in condition (G2) that we assume to hold.
	
The first group, defined by (\ref{Olga1}), includes Proposition \ref{5.6}, Corollary \ref{dhjfsjhdfgj}, and Theorem \ref{5.8}. The second group, determined by (\ref{Olga2}), comprises Proposition \ref{5.6ppp}, Corollary \ref{kekrlkdkf}, and Theorem \ref{5.9}. We conclude this section with the proof of Theorem \ref{belongss}.

\begin{proposition}
\label{5.6} 
Suppose that $F$ satisfies (P1), (P2), (P3), (G1), Equation~(\ref{Olga1}) of (G2), and (G3).  
Then, for every signed measure $\mu \in \mathbf{L}^{\infty}$, it holds 
\begin{equation}
\|\func{\overline{F}}_{\Phi, h} \mu \|_{\infty} \le \alpha \|\mu\|_{\infty} + |\phi_1|_w.
\label{abovv}
\end{equation}
\end{proposition}

\begin{proof}
Let $f_i$ denote the inverse branches of $f$, for $i = 1, \dots, \deg(f)$.  
By Equations~(\ref{gjhfjghfjgj}), (\ref{Olga1}), Corollary~\ref{fff}, and Remark~\ref{yturhfvb}, we have, for each $\gamma \in M$,

\begin{align*}
\|(\func{\overline{F}}_{\Phi,h}\mu)|_\gamma\|_o
&= \Big\|\frac{1}{\lambda}\sum_{i=1}^{\deg(f)}
\frac{\func{F}_{\gamma_i*}(\mu|_{\gamma_i})\,h(\gamma_i)e^{\varphi(\gamma_i)}}{h(\gamma)}\Big\|_o
\\
&\le \frac{1}{\lambda}\sum_{i=1}^{\deg(f)}
\Big\|\frac{\func{F}_{\gamma_i*}(\mu|_{\gamma_i})\,h(\gamma_i)e^{\varphi(\gamma_i)}}{h(\gamma)}\Big\|_o
\\
&\le \sum_{i=1}^{\deg(f)}\bigl(\alpha\|\mu|_{\gamma_i}\|_o+|\phi_1(\gamma_i)|_w\bigr)
\frac{h(\gamma_i)e^{\varphi(\gamma_i)}}{h(\gamma)\lambda}
\\
&\le (\alpha\|\mu\|_\infty+|\phi_1|_w)
\sum_{i=1}^{\deg(f)}\frac{h(\gamma_i)e^{\varphi(\gamma_i)}}{h(\gamma)\lambda}
\\
&= (\alpha\|\mu\|_\infty+|\phi_1|_w)\,
\frac{\func{\overline{L}}_{\varphi,h}(1)(\gamma)}{\lambda}
= \alpha\|\mu\|_\infty+|\phi_1|_w .
\end{align*}Taking the supremum over $\gamma \in M$, we obtain \eqref{abovv}, which concludes the proof.
\end{proof}

By iterating Proposition~\ref{5.6}, we obtain the following corollary.

\begin{corollary}\label{dhjfsjhdfgj}
Suppose that $F$ satisfies (P1), (P2), (P3), (G1), Equation~(\ref{Olga1}) of (G2), and (G3).  
Then, for every signed measure $\mu \in \mathbf{L}^{\infty}$, it holds
\begin{equation*}
\|\func{\overline{F}}_{\Phi,h}^{\,n} \mu \|_\infty
\le \alpha^n \|\mu\|_\infty + \overline{\alpha}_1 |\phi_1|_w,
\end{equation*}
where $\overline{\alpha}_1 := \frac{1}{1 - \alpha}$.
\end{corollary}

\begin{proposition}
	\label{5.6ppp} Suppose that $F$ satisfies (P1), (P2), (P3), (G1), equation (\ref{Olga2}) of (G2) and (G3). Then, for every signed measure $\mu \in \mathbf{L}^{1}$, it holds 
	\begin{equation}
		||\func {\overline{F}}_\Phi\mu ||_{1}\leq \alpha ||\mu ||_{1}+(\alpha +1)|\phi
		_{1}|_{w}.  \label{abovvppp}
	\end{equation}
\end{proposition}

\begin{proof}
	Consider a signed measure $\mu \in \mathbf{L}^{1}$ and its restriction on
	the leaf $\gamma $, $\mu |_{\gamma }=\pi _{\gamma ,2\ast }(\phi _{1}(\gamma
	)\mu _{\gamma })$. Set%
	\begin{equation*}
		\overline{\mu }|_{\gamma }=\pi _{\gamma ,2\ast }\mu _{\gamma }.
	\end{equation*}%
	If $\mu $ is a positive measure then $\overline{\mu }|_{\gamma }$ is a
	probability on $K$ and $\mu |_{\gamma }=\phi _{1}(\gamma )\overline{\mu }%
	|_{\gamma }$. Moreover, let $f_{i}$ be the inverse branches of $f$, for all $i=1\cdots \deg(f)$. Applying Equations (\ref{gjhfjghfjgj}) and (\ref{Olga2}) and the expression given by Proposition \ref{disintnorm}
	we have%
	\begin{equation*}
		\begin{split}
			&||\func {\overline{F}}_\Phi\mu ||_{1} \\
			&\leq \sum_{i=1}^{\deg(f)}{\ \int_{f(P_{i})}{\
					\left\vert \left\vert \frac{\func{F}_{\gamma _i\ast }\overline{\mu
							^{+}}|_{\gamma _i}\phi _{1}^{+}(\gamma _i)}{\lambda}e^{\varphi(\gamma_i)}-\frac{\func{F}_{\gamma _i\ast }%
						\overline{\mu ^{-}}|_{\gamma _i}\phi _{1}^{-}(\gamma _i)%
					}{\lambda}e^{\varphi(\gamma_i)}\right\vert \right\vert _o}%
				d\nu(\gamma )} \\
			&\leq \func{I}_{1}+\func{I}_{2},
		\end{split}
	\end{equation*}%
	where%
	\begin{equation*}
		\func{I}_{1}=\sum_{i=1}^{\deg(f)}{\ \int_{f(P_{i})}{\ \left\vert \left\vert \frac{%
					\func{F}_{\gamma _i\ast }\overline{\mu ^{+}}|_{\gamma _i}\phi _{1}^{+}(\gamma _i)}{\lambda}e^{\varphi(\gamma_i)}
				-\frac{\func{F}_{\gamma _i\ast }\overline{\mu ^{+}}%
					|_{\gamma _i}\phi _{1}^{-}(\gamma _i)}{\lambda}e^{\varphi(\gamma_i)}\right\vert \right\vert _o}d\nu(\gamma )%
		}
	\end{equation*}%
	and%
	\begin{equation*}
		\func{I}_{2}=\sum_{i=1}^{\deg(f)}{ \int_{f(P_{i})}{\ \left\vert \left\vert \frac{%
					\func{F}_{\gamma _i\ast }\overline{\mu ^{+}}|_{\gamma _i}\phi _{1}^{-}(\gamma _i)}{\lambda}e^{\varphi(\gamma_i)}
				-\frac{\func{F}_{\gamma _i\ast }\overline{\mu ^{-}}%
					|_{\gamma _i}\phi _{1}^{-}(\gamma _i)}{\lambda}e^{\varphi(\gamma_i)}\right\vert \right\vert _o}d\nu(\gamma )%
		}.
	\end{equation*}%
	In the following we estimate $\func{I}_{1}$ and $\func{I}_{2}$. By Equation (\ref{Olga2}) of (G2), (G3) and a change of variable ($\beta = f_i(\gamma)$) we have (by P1 the Jacobian of $\nu$ is $\lambda e ^{-\varphi}$)

\begin{eqnarray*}
\func{I}_{1}
&=&\sum_{i=1}^{\deg(f)} \int_{f(P_{i})}
e^{\varphi(\gamma_i)}
\left\Vert 
\func{F}_{\gamma_i\ast }\overline{\mu ^{+}}|_{\gamma_i}
\right\Vert _o
\frac{|\phi _{1}^{+}-\phi _{1}^{-}|( \gamma_i)}{\lambda}
\,d\nu (\gamma ) \\
&=&\int_{M}
\left\Vert
\func{F}_{\beta \ast }\overline{\mu ^{+}}|_{\beta }
\right\Vert _o
|\phi _{1}^{+}-\phi _{1}^{-}|(\beta )
\,d\nu(\beta )
=\int_{M} |\phi _{1}^{+}-\phi _{1}^{-}|(\beta )\,d\nu(\beta ) \\
&=&|\phi _{1}|_{1}
\leq |\phi _{1}|_{w}.
\end{eqnarray*}By (G1.1) and the same change of variables as above, we obtain
	
\begin{eqnarray*}
\func{I}_{2}
&=&\sum_{i=1}^{q} \int_{f(P_{i})}
e^{\varphi(\gamma_i)}
\left\Vert 
\func{F}_{\gamma_i\ast }
\left( \overline{\mu ^{+}}|_{\gamma_i}
-\overline{\mu ^{-}}|_{\gamma_i}\right)
\right\Vert _o
\frac{\phi _{1}^{-}(\gamma_i)}{\lambda}
\,d\nu(\gamma ) \\
&\leq&\sum_{i=1}^{\deg(f)} \int_{P_{i}}
\left\Vert 
\func{F}_{\beta \ast }
\left( \overline{\mu ^{+}}|_{\beta }
-\overline{\mu ^{-}}|_{\beta }\right)
\right\Vert _o
\phi _{1}^{-}(\beta )
\,d\nu(\beta ) \\
&\leq&\alpha \int_{M}
\left\Vert 
\overline{\mu ^{+}}|_{\beta }
-\overline{\mu ^{-}}|_{\beta }
\right\Vert _o
\phi _{1}^{-}(\beta )
\,d\nu(\beta ) \\
&\leq&\alpha \int_{M}
\left\Vert 
\overline{\mu ^{+}}|_{\beta }\phi _{1}^{-}(\beta )
-\overline{\mu ^{-}}|_{\beta }\phi _{1}^{-}(\beta )
\right\Vert _o
\,d\nu(\beta ) \\
&\leq&\alpha \int_{M}
\left\Vert 
\overline{\mu ^{+}}|_{\beta }\phi _{1}^{-}(\beta )
-\overline{\mu ^{+}}|_{\beta }\phi _{1}^{+}(\beta )
\right\Vert _o
\,d\nu(\beta ) \\&+&
\alpha \int_{M}
\left\Vert 
\overline{\mu ^{+}}|_{\beta }\phi _{1}^{+}(\beta )
-\overline{\mu ^{-}}|_{\beta }\phi _{1}^{-}(\beta )
\right\Vert _o
\,d\nu(\beta ) \\
&=&\alpha |\phi _{1}|_{1}+\alpha \|\mu\|_{1}
\;\leq\; \alpha |\phi _{1}|_{w}+\alpha \|\mu\|_{1}.
\end{eqnarray*}

	Summing the above estimates we finish the proof.
\end{proof}

Iterating (\ref{abovvppp}) we get the following corollary.

\begin{corollary}\label{kekrlkdkf}
	Suppose that $F$ satisfies (P1), (P2), (P3), (G1), Equation (\ref{Olga2}) of (G2) and (G3). Then, for every signed measure $\mu \in \mathbf{L}^{1}$ it holds 
	\begin{equation*}
		||\func {\overline{F}}_\Phi^{n}\mu ||_{1}\leq \alpha ^{n}||\mu ||_{1}+\overline{%
			\alpha }_2|\phi _1|_{w},
	\end{equation*}%
	where $\overline{\alpha }_2=\frac{1+\alpha }{1-\alpha }$. \label{nicecoropppp}
\end{corollary}

Now we are ready to prove Theorems \ref{5.8} and \ref{5.9}.

\begin{proof}[Proof of Theorem \ref{5.8}]
Let $\mu \in \mathbf{V}^\infty$. By Equation~\eqref{lasotaiisdffrr} of Remark~\ref{jhdfgjsd}, it holds
\[
|\mathcal{\overline{L}}_{\varphi,h}^{n}(\phi_1)|_s
\le D_3 r_3^{n} |\phi_1|_s,
\qquad \forall n \ge 1.
\]
Since $|\phi_1|_w \le \|\mu\|_{\infty}$, we obtain
\[
|\mathcal{\overline{L}}_{\varphi,h}^{n}(\phi_1)|_s
\le D_3 r_3^{n} \|\mu\|_{\mathbf{S}^{\infty}},
\qquad \forall n \ge 1.
\]Let $l \in \mathbb{N}$ and $d \in \{0,1\}$ be such that $n = 2l + d$. By Proposition~\ref{kdjfjksdkfhjsdfk}, the operator $\overline{F}_{\Phi,h}$ is a weak contraction, hence
\[
\|\overline{F}_{\Phi,h}^{n}\mu\|_{\infty} \le \|\mu\|_{\infty},
\qquad \forall n \ge 1,
\]
and clearly $\|\mu\|_{\infty} \le \|\mu\|_{\mathbf{S}^{\infty}}$.
By Corollary~\ref{nicecoro}, it follows that (setting
$\beta_3 = \max\{\sqrt{r_3}, \sqrt{\alpha}\}$)
\begin{eqnarray*}
\|\overline{F}_{\Phi,h}^{n}\mu\|_{\infty}
&=& \|\overline{F}_{\Phi,h}^{2l+d}\mu\|_{\infty} \\
&\le& \alpha^{l}
\|\overline{F}_{\Phi,h}^{l+d}\mu\|_{\infty}
+ \overline{\alpha}_1
\left|
\mathcal{\overline{L}}_{\varphi,h}^{\,l+d}(\phi_1)
\right|_w \\
&\le& \alpha^{l}\|\mu\|_{\infty}
+ \overline{\alpha}_1
\left|
\mathcal{\overline{L}}_{\varphi,h}^{\,l+d}(\phi_1)
\right|_w \\
&\le& \alpha^{l}\|\mu\|_{\infty}
+ \overline{\alpha}_1
\left|
\mathcal{\overline{L}}_{\varphi,h}^{\,l}(\phi_1)
\right|_s \\
&\le&
\bigl(
\sqrt{\alpha}^{-1}
+ \overline{\alpha}_1 D_3 \sqrt{r_3}^{-1}
\bigr)
\beta_3^{n}
\|\mu\|_{\mathbf{S}^{\infty}} \\
&\le& D \beta_3^{n} \|\mu\|_{\mathbf{S}^{\infty}},
\end{eqnarray*}
where $D =\sqrt{\alpha}^{-1}+ \overline{\alpha}_1 D_3 \sqrt{r_3}^{-1}.$
\end{proof}The proof of Theorem~\ref{5.9} is analogous and is therefore omitted.

\begin{proof}[Proof of Theorem \ref{belongss}]

By (P3) and (P1), let $\mu _{0}$ be the $F$-invariant measure such that $\pi _{1\ast }\mu _{0}=m$, where $m=h \nu$. Thus, by (G3) it holds $\left\vert \left\vert \mu _{0}|_{\gamma}\right\vert \right\vert _o  = |h(\gamma)|$, since $\dfrac{\pi_{1*}\mu_0}{d \nu} \equiv h$. By (G2) the proof is complete.
	
	The uniqueness follows directly from Theorem \ref{5.8} and \ref{5.9}, since the difference between two probabilities ($\mu _1 - \mu_0$) is a zero average signed measure.
\end{proof}
\subsection{Spectral Gap}\label{jshdjfgsjhdf}
\begin{proof}[Proof of Theorem~\ref{spgap}]
We first show that there exist constants $0<\xi<1$ and $R_{1}>0$ such that, for all $n\ge1$,
\begin{equation}
\|\func{\overline{F}}_\Phi^{n}\|_{\mathbf{V}^1\rightarrow \mathbf{V}^1}
\leq R_{1}\xi^{n},
\label{quaselawww}
\end{equation}
where $\mathbf{V}^1$ denotes the zero-average subspace defined in~\eqref{mathV}.
Fix $\mu \in \mathbf{V}^1$ such that $\|\mu\|_{\mathbf{S}^1}\le1$.
For a given $n\in\mathbb{N}$, write $n=2m+d$, where $m=\frac{n-d}{2}$ and $d\in\{0,1\}$.

By the Lasota--Yorke inequality (Proposition~\ref{lasotaoscilation2}), there exists a uniform bound
\[
\|\func{\overline{F}}_\Phi^{n}\mu\|_{\mathbf{S}^{1}}\le A+B_{2},
\quad \forall\, n\ge1.
\]
Let $\beta_{0}=\max\{\beta_{2},\beta_{4}\}$. Then
\begin{eqnarray*}
\|\func{\overline{F}}_\Phi^{n}\mu\|_{\mathbf{S}^{1}}
&\leq& A\beta_{2}^{m}\|\func{\overline{F}}_\Phi^{m+d}\mu\|_{\mathbf{S}^{1}}
+ B_{2}\|\func{\overline{F}}_\Phi^{m+d}\mu\|_{1} \\
&\leq& A\beta_{2}^{m}(A+B_{2}) + B_{2}D_{4}\beta_{4}^{m} \\
&\leq& \beta_{0}^{m}\bigl[A(A+B_{2})+B_{2}D_{4}\bigr] \\
&\leq& \xi^{n}R_{1},
\end{eqnarray*}
Recall that $\func{\overline{F}}_\Phi:\mathbf{S}^{1}\to\mathbf{S}^{1}$ admits a unique fixed point $\mu_{0}\in\mathbf{S}^{1}$, which is a probability measure.
Define the projection $\func{P}:\mathbf{S}^{1}\to \mathrm{span}\{\mu_{0}\}$ by
\[
\func{P}(\mu)=\mu(\Sigma)\,\mu_{0}.
\]
Setting $\func{N}=\func{\overline{F}}_\Phi\circ\func{S}$, we obtain
$\func{\overline{F}}_\Phi=\func{P}+\func{N}$, with $\func{P}\func{N}=\func{N}\func{P}=0$. Then, $\func{N}^n=\func{\overline{F}}_\Phi^n\circ\func{S}$ for all $n\geq1$.
Moreover, since $\func{S}$ is bounded and $\func{S}(\mu)\in\mathbf{V}^1$, estimate~\eqref{quaselawww} yields
\[
\|\func{N}^{n}(\mu)\|_{\mathbf{S}^{1}}
\le R\xi^{n}\|\mu\|_{\mathbf{S}^{1}},
\quad \forall\, n\ge1,
\]
where $R=R_{1}\|\func{S}\|_{\mathbf{S}^{1}\to\mathbf{S}^{1}}$.
\end{proof}The proof of Theorem~\ref{spgapp} follows the same strategy and is therefore omitted.

\subsection{Exponential Decay of Correlations}\label{dfjgsghdfjasdf}

In this subsection we establish exponential decay of correlations for the dynamical system $F$ with respect to the distinguished invariant
probability measure $\mu_0$ provided by assumption~(P3).

More precisely, we prove that correlations between pairs of observables
$u$ and $v$, belonging respectively to the classes $B'_i$ and
$\Theta_{\mu_0}^i$ for $i \in \{1,\infty\}$ (introduced before the statements of
Theorems~\ref{shkjfjdhsf} and~\ref{slkdgjsdg}), decay exponentially fast.
The result follows as a direct consequence of the spectral gap properties
of the transfer operators $\func{\overline{F}}_\Phi$ acting on the normed spaces
$\mathbf{S}^1$ and $\mathbf{S}^\infty$.

We present the proof of Theorem~\ref{shkjfjdhsf}, corresponding to the
$\mathbf{S}^1$ setting. The proof of Theorem~\ref{slkdgjsdg} follows by the same
argument, replacing $\mathbf{S}^1$ with $\mathbf{S}^\infty$ and using Theorem \ref{spgapp} and Equation (\ref{poyuoypuiyu}).

\begin{proof}[Proof of Theorem~\ref{shkjfjdhsf}]
Let $u \in B'_1$ and $v \in \Theta_{\mu_0}^1$.  
By Theorem~\ref{spgap} and Proposition~\ref{olgaa1}, we obtain
\begin{align*}
\left| \int (u \circ F^n) v \, d\mu_0 - \int u \, d\mu_0 \int v \, d\mu_0 \right|
&= \left| \int u \, d\func{\overline{F}}_\Phi^{n}(v\mu_0)
      - \int u \, d\func{P}(v\mu_0) \right| \\
&\le R_3 \, \big\| \func{\overline{F}}_\Phi^{n}(v\mu_0)
      - \func{P}(v\mu_0) \big\|_{\mathbf{S}^1} \\
&= R_3 \, \big\| \func{N}^n (v\mu_0) \big\|_{\mathbf{S}^1} \\
&\le R_3 R \, \xi^n \, \| v \mu_0 \|_{\mathbf{S}^1},
\end{align*}which concludes the proof.
\end{proof}

\section{Application on Piecewise Partially Hyperbolic Maps}\label{appppl}
\subsection{The dynamics}

Let $F:\Sigma \longrightarrow \Sigma$ be the map defined by
\begin{equation*}
	F(x,z) = (f(x), G(x,z)),
\end{equation*}
where $G:\Sigma \longrightarrow K$ and $f:M \longrightarrow M$ are measurable functions satisfying the assumptions below.

\subsubsection{Hypotheses on $f$}\label{hf}

Assume that $f:M \longrightarrow M$ is a local homeomorphism. Suppose there exists a continuous function $L:M \longrightarrow \mathbb{R}$ such that, for every $x \in M$, there is a neighborhood $U_x$ of $x$ for which the restriction
\[
f_x := f|_{U_x} : U_x \longrightarrow f(U_x)
\]
is invertible and satisfies
\begin{equation}\label{iirotyirty}
	d_1\bigl(f_x^{-1}(y), f_x^{-1}(z)\bigr) \leq L(x)\, d_1(y,z),
	\quad \forall\, y,z \in f(U_x).
\end{equation}In particular, the cardinality $\# f^{-1}(x)$ is constant for all $x \in M$. We denote this constant by $\deg(f)$.

Assume further that there exist an open set $\mathcal{A} \subset M$ and constants $\sigma > 1$ and $L \geq 1$ such that:
\begin{enumerate}
	\item[(f1)] $L(x) \leq L$ for all $x \in \mathcal{A}$, and $L(x) < \sigma^{-1}$ for all $x \in \mathcal{A}^c$. Moreover, $L$ is sufficiently close to $1$ (a precise bound is given in Equation \eqref{kdljfhkdjfkasd});
	\item[(f2)] There exists a finite open covering $\mathcal{U}$ of $M$ by domains of injectivity of $f$ such that $\mathcal{A}$ can be covered by $q < \deg(f)$ elements of $\mathcal{U}$.
\end{enumerate}

Condition (f1) allows expanding and contracting behavior to coexist on $M$: the map $f$ is uniformly expanding outside $\mathcal{A}$ and not too contracting inside $\mathcal{A}$. If $\mathcal{A} = \varnothing$, then $f$ is uniformly expanding. Condition (f2) ensures that every point has at least one preimage in the expanding region.

\begin{definition}\label{ejhrkjhe}
For $0 < \zeta \leq 1$, let $H_\zeta$ denote the space of $\zeta$-H\"older continuous functions $h:M \longrightarrow \mathbb{R}$, endowed with the seminorm
\[
H_\zeta(h) := \sup_{x \neq y} \frac{|h(x) - h(y)|}{d_1(x,y)^\zeta}.
\]
That is,
\[
H_\zeta := \bigl\{ h:M \longrightarrow \mathbb{R} \; : \; H_\zeta(h) < \infty \bigr\}.
\]
\end{definition}

\begin{definition}\label{PH}
Let $\mathscr{P}_M$ be the set of H\"older potentials satisfying the following condition, which is open with respect to the H\"older norm:
\begin{enumerate}
	\item[(f3)] There exists $\epsilon_\varphi > 0$ sufficiently small such that
	\begin{equation}\label{f31}
		\sup \varphi - \inf \varphi < \epsilon_\varphi,
	\end{equation}
	and
	\begin{equation}\label{f32}
		H_\zeta(e^\varphi) < \epsilon_\varphi\, e^{\inf \varphi}.
	\end{equation}
\end{enumerate}
We assume that the constants $\epsilon_\varphi$ and $L$ satisfy
\begin{equation}\label{kdljfhkdjfkasd}
	\exp(\epsilon_\varphi)\,
	\frac{(\deg(f)-q)\sigma^{-\alpha} + q L^\alpha \bigl[1 + (L-1)^\alpha\bigr]}
	{\deg(f)} < 1.
\end{equation}
\end{definition}Condition \eqref{f32} means that $\varphi$ belongs to a small cone of H\"older continuous functions (see \cite{VAC}).

According to \cite{VAC}, any map $f:M \longrightarrow M$ satisfying (f1), (f2), and (f3) admits an invariant probability measure $m$ of maximal entropy, which is absolutely continuous with respect to a conformal measure $\nu$. Moreover, for every $\varphi \in \mathscr{P}_M$, the normalized Ruelle--Perron--Frobenius operator $\overline{\mathcal{L}}_\varphi$ (see \eqref{hhdfhjdghf}) satisfies property (P1), with $(B,|\cdot|_b) = (H_\zeta,|\cdot|_\zeta)$ and $(B_w,|\cdot|_w) = (C^0,|\cdot|_\infty)$. Since $M$ is compact, property (P2) also holds (see Lemma~2.2 in \cite{RRR}).

\subsubsection{Hypotheses on $G$}\label{sjdfgjsdhf}

Assume that $G:\Sigma \longrightarrow K$ satisfies the following condition:
\begin{enumerate}
	\item[(H1)] $G$ is uniformly contracting along $\nu$-almost every vertical fiber
	\[
	\gamma_x := \{x\} \times K.
	\]
	That is, there exists a constant $0 \leq \alpha < 1$ such that, for $\nu$-almost every $x \in M$, one has
	\begin{equation}\label{contracting1}
		d_2\bigl(G(x,z_1), G(x,z_2)\bigr)
		\leq \alpha\, d_2(z_1,z_2),
		\quad \forall\, z_1,z_2 \in K.
	\end{equation}
\end{enumerate}

We denote by $\mathcal{F}$ the family of all vertical fibers, namely
\[
\mathcal{F} := \{ \gamma_x := \{x\} \times K \; ; \; x \in M \}.
\]
When no confusion arises, elements of $\mathcal{F}$ will be denoted simply by $\gamma$.

Proposition~\ref{kjdhkskjfkjskdjf} below ensures the existence and uniqueness of an $F$-invariant probability measure $\mu_0$ projecting onto $m$. Since its proof follows standard arguments (see, for instance, \cite{AP}), we omit it. Consequently, if $F$ satisfies conditions (f1), (f2), (f3), and (H1), then property (P3) holds.

\begin{proposition}\label{kjdhkskjfkjskdjf}
	Let $m_1$ be an $f$-invariant probability. If $F$ satisfies (H1), then there exists an unique measure $\mu_1$ on $M \times K$ such that $\pi_1{_\ast}\mu_1 = m_1$ and for every continuous function $\psi \in C^0 (M \times K)$ it holds 
	\begin{equation*}
		\lim {\int{\inf_{\gamma \times K} \psi \circ F^n }dm_1(\gamma)}= \lim {\int{\sup_{\gamma \times K} \psi \circ F^n}dm_1 (\gamma)}=\int {\psi}d\mu_1. 
	\end{equation*}Moreover, the measure $\mu_1$ is $F$-invariant.
\end{proposition}

\begin{remark}\label{pot}
Observe that the probability measure $m$ is an equilibrium state for a potential
$\varphi \in \mathscr{P}_M$. Consequently, the invariant measure $\mu_0$ depends on
the function $\Phi := \varphi \circ \pi_1 \in \mathscr{P}_\Sigma$
(see Equation~(\ref{PPP})).
\end{remark}

\subsubsection{The $||\cdot||_o$ norm}

In this section, we define a norm satisfying condition (G1).

\begin{definition}\label{wasserstein}
Given two signed measures $\mu$ and $\nu$ on $K$, we define a
\textbf{Wasserstein--Kantorovich-like} distance between $\mu$ and $\nu$ as follows
(cf.\ Definition~\ref{ejhrkjhe}):
\[
	W_{1}^{\zeta}(\mu,\nu)
	:= \sup_{\substack{H_\zeta(g)\leq 1\\ |g|_{\infty}\leq 1}}
	\left| \int g\,d\mu - \int g\,d\nu \right|.
\]
\end{definition}

Since the constant $0 < \zeta \leq 1$ is fixed, we shall henceforth write
\[
	\|\mu\|_o := W_{1}^{\zeta}(0,\mu).
\] In fact, $\|\cdot\|_o$ defines a norm on the vector space of signed measures on a
compact metric space, which is equivalent to the dual norm of the space of
$\zeta$-H\"older continuous functions.

Other applications of particular cases of this metric to statistical properties
can be found in \cite{RRR, RRR2, GLu, ben, LiLu}. For instance, in \cite{ben} the
author applies this metric to a more general class of shrinking fibers systems,
whereas in \cite{RRR2} a quantitative statistical stability result is obtained.

The next lemmas, Lemmas~\ref{niceformulaac} and~\ref{opsdas}, yield that if we set
$\|\cdot\|_o := W_1^\zeta(0,\cdot)$, then any map $F$ satisfying
(f1), (f2), (f3), and (H1) also satisfies conditions (G1) and (G3). Moreover, since
\[
|\phi_1(\gamma)|
= |\mu|_\gamma(K)|
= \left| \int 1 \, d\mu|_\gamma \right|
\leq \|\mu|_\gamma\|_o,
\]
we conclude that $F$ satisfies Equations~(\ref{Olga1}) and~(\ref{Olga2}) of (G2).

Other applications of particular cases of this metric on statistical properties can be seen in \cite{RRR}, \cite{RRR2}, \cite{GLu}, \cite{ben} and \cite {LiLu}. For instance, in \cite {ben} the author apply this metric to a more general case of shrinking fibers systems. While in \cite{RRR2} a quantitative statistical stability statement was obtained. 

Next Lemmas \ref{niceformulaac} and \ref{opsdas} yields that, if we set $||\cdot ||_o:= W^\zeta _1(0, \cdot)$, then a map $F$ which satisfies (f1), (f2), (f3) and (H1), also satisfies (G1) and (G3). Moreover, since $|\phi_1 (\gamma)| = |\mu |_\gamma (K)|= |\int 1 d\mu |_\gamma |\leq ||\mu|_\gamma||_o$ we have that $F$ satisfies (\ref{Olga1}) and (\ref{Olga2}) of (G2).

\begin{lemma}
	\label{niceformulaac} For every $\mu \in \mathbf{AB}_\nu$ and $\nu$-a.e. stable leaf $%
	\gamma \in \mathcal{F}$, it holds 
	\begin{equation}
		||\func{F}_{\gamma \ast }\mu |_{\gamma }||_o\leq ||\mu |_{\gamma }||_o,
		\label{weak1}
	\end{equation}%
	where $F_{\gamma }:K\longrightarrow K$ is defined in Equation (\ref{ritiruwt}) and $\func{F}_{\gamma \ast }$ is the associated pushforward
	map. Moreover, if $\mu $ is a probability measure on $K$, for $\nu$-a.e. stable leaf $\gamma \in \mathcal{F}$ it holds 
	\begin{equation}
		||\func{F}_{\gamma \ast}^{n}\mu ||_o=||\mu ||_o=1,\ \ \forall \ \
		n\geq 1.  \label{simples}
	\end{equation}
\end{lemma}

\begin{proof}
Since $F_{\gamma}$ is an $\alpha$-contraction, if $|g|_{\infty} \leq 1$ and
$H_\zeta(g) \leq 1$, then the same bounds hold for $g \circ F_{\gamma}$. Moreover,
\begin{equation*}
	\left| \int g \, d\func{F}_{\gamma\ast}\mu|_{\gamma} \right|
	= \left| \int g \circ F_{\gamma} \, d\mu|_{\gamma} \right|.
\end{equation*}
Taking the supremum over all functions $g$ such that $|g|_{\infty} \leq 1$ and
$H_\zeta(g) \leq 1$, we conclude the proof of inequality~(\ref{weak1}).

To prove Equation~(\ref{simples}), let $\mu$ be a probability measure on $K$ and
let $g : K \to \mathbb{R}$ be a $\zeta$-H\"older function with
$\|g\|_{\infty} \leq 1$. Then
\[
\left| \int g \, d\mu \right| \leq \|g\|_{\infty} \leq 1,
\]
which implies $\|\mu\|_o \leq 1$. On the other hand, choosing $g \equiv 1$, we obtain
$\|\mu\|_o = 1$, concluding the proof.
\end{proof}

\begin{lemma}\label{opsdas}
	For all signed measures $\mu $ on $K$ and for $\nu$-a.e. $\gamma \in \mathcal{F}%
	$, it holds%
	\begin{equation*}
		||\func{F}_{\gamma \ast }\mu ||_o\leq \alpha^\zeta ||\mu ||_o+|\mu (K)|
	\end{equation*}%
	($\alpha $ is the rate of contraction of $G$, see \eqref{contracting1}). In
	particular, if $\mu (K)=0$ then%
	\begin{equation*}
		||\func{F}_{\gamma \ast }\mu ||_o\leq \alpha^\zeta ||\mu ||_o.
	\end{equation*}%
	\label{quasicontract}
\end{lemma}

\begin{proof}
	If $H_\zeta(g)\leq 1$ and $||g||_{\infty }\leq 1$, then $g\circ F_{\gamma }$ is $%
	\alpha^\zeta $-H\"older. Moreover, since $||g||_{\infty }\leq 1$, then $||g\circ
	F_{\gamma }-\theta ||_{\infty }\leq \alpha^\zeta $, for some $\theta $ such that $%
	|\theta |\leq 1$. Indeed, let $z\in K$ be such that $|g\circ F_{\gamma
	}(z)|\leq 1$, set $\theta =g\circ F_{\gamma }(z)$ and let $d_{2}$ be the
 metric of $K$. Since $\diam(K)=1$, we have 
	\begin{equation*}
		|g\circ F_{\gamma }(y)-\theta |\leq \alpha^\zeta d_{2}(y,z)\leq \alpha^\zeta
	\end{equation*}%
	and consequently $||g\circ F_{\gamma }-\theta ||_{\infty }\leq \alpha^\zeta $.
	
	This implies
	
	\begin{align*}
		\left\vert \int_{K}{g}d\func{F}_{\gamma \ast }\mu \right\vert &
		=\left\vert \int_{K}{g\circ F_{\gamma }}d\mu \right\vert \\
		& \leq \left\vert \int_{K}{g\circ F_{\gamma }-\theta }d\mu \right\vert
		+\left\vert \int_{K}{\theta }d\mu \right\vert \\
		& =\alpha^\zeta \left\vert \int_{K}{\frac{g\circ F_{\gamma }-\theta }{\alpha^\zeta }}%
		d\mu \right\vert +|\theta ||\mu (K)|.
	\end{align*}%
	Taking the supremum over $g$ such that $|g|_{\infty }\leq 1$ and $%
	H_\zeta(g)\leq 1$, we have $||\func{F}_{\gamma \ast }\mu ||_o\leq \alpha^\zeta ||\mu
	||_o+|\mu (K)|$. In particular, if $\mu (K)=0$, we get the second
	part.
\end{proof}

\subsubsection{Consequences for the dynamics of $F$}

With the following definitions in place, all the results of
Section~\ref{kjdfkjdsfkj} apply to the map $F$.

\begin{definition}\label{P}
Define
\begin{equation*}
	\mathscr{P}_\Sigma
	:= \left\{ \Phi \; ; \; \Phi = \varphi \circ \pi_1,
	\ \varphi \in \mathscr{P}_M \right\}.
\end{equation*}
\end{definition}

\begin{definition}\label{s11112}
Define the vector space $\mathbf{S}^1$ by
\begin{equation*}
	\mathbf{S}^{1}
	= \left\{ \mu \in \mathbf{L}^{1} \; ; \; \phi_{1} \in H_\zeta \right\},
\end{equation*}
and equip it with the norm
$ \|\cdot\|_{\mathbf{S}^{1}} : \mathbf{S}^{1} \to \mathbb{R} $ given by
\begin{equation*}
	\|\mu\|_{\mathbf{S}^{1}}
	:= |\phi_{1}|_\zeta + \|\mu\|_{1}.
\end{equation*}
\end{definition}

\begin{definition}
Define $\mathbf{S}^\infty$ by
\begin{equation*}
	\mathbf{S}^{\infty}
	= \left\{ \mu \in \mathbf{L}^{\infty} \; ; \; \phi_{1} \in H_\zeta \right\},
\end{equation*}
and equip it with the norm
$ \|\cdot\|_{\mathbf{S}^{\infty}} : \mathbf{S}^{\infty} \to \mathbb{R} $ given by
\begin{equation*}
	\|\mu\|_{\mathbf{S}^{\infty}}
	:= |\phi_{1}|_\zeta + \|\mu\|_{\infty}.
\end{equation*}
\end{definition}

\subsubsection{Equilibrium States and the Proof of Theorem \ref{belongsss}}

\begin{definition}
	Let $F:\Sigma \longrightarrow \Sigma$ be a measurable map and $\Phi \in C^{0}(\Sigma, \mathbb{R})$ (continuous functions space), define the {\bf pressure} of $\Phi$ by
	\begin{equation*}
		\label{pmens}
		P(F,\Phi) :=  \sup_{\mu \in \mathcal{M}^{1}_{F}(\Sigma)} \left\{ h_{\mu}(F) + \int \Phi d \mu \right\},
	\end{equation*}
	where $\mathcal{M}^{1}_{F}(\Sigma)$ denotes set of $F$- invariant probabilities on $\Sigma$. We say that $\mu \in \mathcal{M}^{1}_{F}(\Sigma)$ is an {\bf equilibrium state} for $\Phi$ if  
\begin{equation*}
\label{pmensrt}
P(F,\Phi) =   h_{\mu}(F) + \int \Phi d \mu.
\end{equation*}
\end{definition}Analogously, denote by $\mathcal{M}^{1}_{f}(M)$ the set of $f$-invariant probabilities.

\begin{remark}\label{kahdhahfjsdf}
	Let $\Phi \in \mathscr{P}_\Sigma$ be a potential of the form
	$\Phi = \varphi \circ \pi_1$, where $\varphi \in \mathscr{P}_M$.
	Suppose that $\mu$ is a probability measure on $\Sigma$ and that
	$\lambda$ is a probability measure on $M$ such that
	$\pi_{1*}\mu = \lambda$.
	Then
	\begin{equation}\label{hjdhfdjfhlkjlggg}
		\int \Phi \, d\mu = \int \varphi \, d\lambda.
	\end{equation}
	Moreover, by Proposition~\ref{kjdhkskjfkjskdjf}, there is a one-to-one
	correspondence between the sets $\mathcal{M}_f^1(M)$ and
	$\mathcal{M}_F^1(\Sigma)$.
	Finally, if $F$ uniformly contracts all fibres, then by
	Abramov--Rokhlin's formula (see \cite{BC92}) and since $G$ is uniformly
	contracting along the fibres, we have
	\begin{equation}\label{pressfactor}
		\sup_{\mu \in \mathcal{M}_F^1(\Sigma)}
		\left\{ h_\mu(F) + \int \Phi \, d\mu \right\}
		=
		\sup_{\lambda \in \mathcal{M}_f^1(M)}
		\left\{ h_\lambda(f) + \int \varphi \, d\lambda \right\}.
	\end{equation}
\end{remark}

\begin{proof}[Proof of Theorem \ref{belongsss}]
For a given potential $\Phi \in \mathscr{P}_\Sigma$, let $\mu _{0}$ be the $F$-invariant measure such that $\pi _{1\ast }\mu _{0}=m$ (which do exist by Proposition \ref{kjdhkskjfkjskdjf}), where $m$ is an equilibrium state for $\varphi$ ($\Phi = \varphi \circ \pi _1$, $\varphi \in \mathscr{P}_M$), such that $m=h \nu$ where $\nu$ is a conformal measure. Suppose that $%
g:K\longrightarrow \mathbb{R}$ is a $\zeta$-H\"older function such that $%
|g|_{\infty }\leq 1$ and $H_\zeta(g)\leq 1$. Then, it holds $\left\vert \int {g} d(\mu _{0}|_{\gamma })\right\vert \leq |g|_{\infty } |h(\gamma)|\leq |h(\gamma)|$. Hence, $\mu _{0}\in \mathbf{L}^{\infty }$. Since, $\dfrac{\pi_{1*}\mu_0}{d\nu} \equiv h$, we have $\mu_0 \in \mathbf{S}^\infty$ and also $\mu_0 \in \mathbf{S}^1$.
The uniqueness follows directly from Theorems \ref{5.8} and \ref{5.9}, since the difference between two probabilities ($\mu _1 - \mu_0$) is a zero average signed measure.

Now suppose that $F$ uniformly contracts all fibres. To prove that $\mu_0$ is an equilibrium state, note first that, since $f$ is a factor of $F$, we have
\begin{equation}\label{dosif}
	h_\lambda(f) \leq h_\mu(F), \quad \forall\, \mu \in \mathcal{M}_F^1(\Sigma) \ \textnormal{such that } \pi_{1*}\mu = \lambda .
\end{equation}

Finally, by Equation~\eqref{dosif}, Remark~\ref{kahdhahfjsdf}, and the fact that $m$ is an equilibrium state for $(f,\varphi)$, we obtain
\begin{eqnarray*}
	P(F,\Phi)
	&=& \sup_{\mu \in \mathcal{M}^{1}_{F}(\Sigma)} \left\{ h_{\mu}(F) + \int \Phi \, d\mu \right\} \\
	&=& \sup_{\lambda \in \mathcal{M}^{1}_{f}(\Sigma)} \left\{ h_\lambda(f) + \int \varphi \, d\lambda \right\} \\
	&=& h_m(f) + \int \varphi \, dm \\
	&\leq& h_{\mu_0}(F) + \int \Phi \, d\mu_0 .
\end{eqnarray*}
This concludes the proof.

\end{proof}

\subsection{H\"older regularity of the invariant measures}

In this section, we assume that the fiber map $G$ satisfies the additional property (H2) stated below, and we prove that the disintegration of every invariant measure $\mu_0$, lifted from $m$, is H\"older regular.

In addition to satisfying Equation~\eqref{kdljfhkdjfkasd}, the constant $L$ appearing in assumptions (f1) and (f3) is also required to satisfy
\begin{equation}
\label{ksjdfkjfdshf}
	(\alpha \cdot L)^\zeta < 1 .
\end{equation}
This condition is not restrictive for the class of dynamics under consideration and is satisfied by all examples described in Section~\ref{dkjfhksjdhfksdf}.

\begin{enumerate}
	\item[(H2)] 
	Let $P_1, \ldots, P_{\deg(f)}$ be the partition of $M$ given by (P2). Assume that
	\begin{equation}
	\label{oityy}
		|G_i|_\zeta
		:= \sup_{y}
		\sup_{x_1, x_2 \in P_i}
		\frac{d_2\!\left(G(x_1,y), G(x_2,y)\right)}
		{d_1(x_1,x_2)^\zeta}
		< \infty .
	\end{equation}
\end{enumerate}

We denote by $|G|_\zeta$ the constant
\begin{equation}
\label{jdhfjdh}
	|G|_\zeta := \max_{i=1,\ldots,\deg(f)} |G_i|_\zeta .
\end{equation}

\begin{remark}
	The condition (H2) allows the map $G$ to be discontinuous on the sets
	$\partial P_i \times K$ for all $i=1,\dots,\deg(f)$, where $\partial P_i$
	denotes the boundary of $P_i$.
\end{remark}

We have observed that a positive measure $\mu \in \mathbf{AB}_m$ can be
disintegrated along the stable leaves $\mathcal{F}$ in such a way that it
can be viewed as a family of positive measures on $K$, denoted by
$\{\mu|_\gamma\}_{\gamma \in \mathcal{F}}$.
Since there exists a one-to-one correspondence between $\mathcal{F}$ and $M$,
this disintegration naturally defines a map from $M$ into the metric space
$\mathcal{SM}(K)$ of positive measures on $K$, endowed with the
Wasserstein--Kantorovich-like metric (see Definition~\ref{wasserstein}).

For convenience, we represent this map using functional notation and denote it by
\[
	\Gamma_\mu : M \longrightarrow \mathcal{SM}(K),
\]
which is defined $m$-almost everywhere by
\[
	\Gamma_\mu(\gamma) = \mu|_\gamma,
\]
where $(\{\mu_\gamma\}_{\gamma \in M}, \phi_1)$ is a disintegration of $\mu$.

Since such a disintegration is only defined for $m$-almost every
$\gamma \in M$, the map $\Gamma_\mu$ is not uniquely defined pointwise.
Therefore, we interpret $\Gamma_\mu$ as an equivalence class of
$m$-almost everywhere equal maps from $M$ to $\mathcal{SM}(K)$ associated
with the measure $\mu$.

\begin{definition}
	Let $\mu \in \mathbf{AB}_m$ be a positive Borel measure and let
	$\omega = (\{\mu_\gamma\}_{\gamma \in M}, \phi_1)$ be a disintegration of $\mu$,
	where $\{\mu_\gamma\}_{\gamma \in M}$ is a family of probability measures on $K$
	defined for $m$-almost every $\gamma \in M$, and $\phi_1$ is a non-negative
	marginal density on $M$, satisfying
	\[
	\pi_{1*}\mu = \phi_1 \, m.
	\]Denote by $\Gamma_\mu$ the equivalence class of maps associated with $\mu$,
	given by
	\begin{equation*}
		\Gamma_\mu := \{ \Gamma_\mu^\omega \}_{\omega},
	\end{equation*}
	where $\omega$ ranges over all possible disintegrations of $\mu$ and
	$\Gamma_\mu^\omega : M \longrightarrow \mathcal{SM}(K)$ is the map associated
	with the disintegration $\omega$, defined by
	\[
		\Gamma_\mu^\omega(\gamma)
		:= \mu|_\gamma
		= \pi _{\gamma, 2} ^\ast \phi_1(\gamma)\, \mu_\gamma .
	\]
\end{definition}

For a given representative $\Gamma_\mu^\omega$, we denote by
$I_{\Gamma_\mu^\omega} \subset M$ the subset of points $\gamma \in M$ on which
$\Gamma_\mu^\omega$ is defined.

\begin{definition}For a given $0<\zeta \leq 1$, a disintegration $\omega$ of $\mu$  and its functional representation $\Gamma_{\mu }^\omega $ we define the \textbf{$\zeta$-H\"older constant of $\mu$ associated to $\omega$} by (the essential supremum below is taken with respect to $m$)

	\begin{equation*}
    \label{Lips1}
		|\mu|_\zeta ^\omega := \esssup _{\gamma_1, \gamma_2 \in I_{\Gamma_{\mu }^\omega}} \left\{ \dfrac{||\mu|_{\gamma _1}- \mu|_{\gamma _2}||_o}{d_1 (\gamma _1, \gamma _2)^\zeta}\right\}.
	\end{equation*}
    Finally, we define the \textbf{$\zeta$-H\"older constant} of the positive measure $\mu$ by

	\begin{equation}\label{Lips2}
		|\mu|_\zeta :=\displaystyle{\inf_{ \Gamma_{\mu }^\omega \in \Gamma_{\mu } }\{|\mu|_\zeta ^\omega\}}.
	\end{equation}
	
	\label{Lips3}
\end{definition}

\begin{remark}
	When no confusion is possible, to simplify the notation, we denote $\Gamma_{\mu }^\omega (\gamma )$ just by $\mu |_{\gamma } $.
\end{remark}

\begin{definition}
	From the Definition \ref{Lips3} we define the set of the $\zeta$-H\"older positive measures $\mathcal{H} _\zeta^{+}$ as
	
	\begin{equation*}
		\mathcal{H} _\zeta^{+}=\{\mu \in \mathbf{L}^\infty:\mu \geq 0,|\mu |_\zeta <\infty \}.
	\end{equation*}
\end{definition}

For the next lemma, for a given path $\Gamma _\mu$ which represents the measure $\mu$, we define for each $\gamma \in I_{\Gamma_{\mu }^\omega }\subset M$, the map

\begin{equation*}
	\mu _F(\gamma) := \func{F_\gamma }_*\mu|_\gamma,
\end{equation*}
where $F_\gamma :K \longrightarrow K$ is defined as

\begin{equation}\label{poier}
	F_\gamma (y) = \pi_2 \circ F \circ {(\pi _2|_\gamma)} ^{-1}(y)
\end{equation}and $\pi_2 : M\times K \longrightarrow  K$ is the projection $\pi_2(x,y)=y$.

\begin{lemma}\label{apppoas}
	Suppose that $F:\Sigma \longrightarrow \Sigma$ satisfies (H1) and (H2). Then, for all $\mu \in \mathcal{H} _\zeta^{+} $ which satisfy $\phi _1 = 1$, $m$-a.e., it holds $$||\func{F}%
	_{x  \ast }\mu |_{x  } - \func{F}%
	_{y \ast }\mu |_{y  }||_o \leq \alpha^\zeta |\mu|_\zeta  d_1(x, y)^\zeta  + |G|_\zeta d_1(x, y)^\zeta ||\mu ||_\infty,$$ for all $x,y \in P_i$ and all $i=1, \cdots, \deg(f)$.
\end{lemma}

\begin{proof}

	Since $(\mu|_x - \mu|_y)(K)=0$ ($\phi _1 = 1$ $m$-a.e.), by Lemma \ref{opsdas}, it holds
	\begin{eqnarray*}
		||\func{F}%
		_{x  \ast }\mu |_{x  } - \func{F}%
		_{y \ast }\mu |_{y  }||_o &\leq & ||\func{F}%
		_{x  \ast }\mu |_{x  } - \func{F}%
		_{x \ast }\mu |_{y  }||_o + ||\func{F}%
		_{x  \ast }\mu |_{y  } - \func{F}%
		_{y \ast }\mu |_{y  }||_o
		\\&\leq & \alpha^\zeta||\mu |_{x  } - \mu |_{y }||_o + ||\func{F}%
		_{x  \ast }\mu |_{y  } - \func{F}%
		_{y \ast }\mu |_{y  }||_o
		\\&\leq & \alpha^\zeta |\mu|_\zeta d_1(x,y)^\zeta + ||\func{F}%
		_{x  \ast }\mu |_{y  } - \func{F}%
		_{y \ast }\mu |_{y  }||_o.
	\end{eqnarray*}Let us estimate the second summand $||\func{F}%
	_{x  \ast }\mu |_{y  } - \func{F}%
	_{y \ast }\mu |_{y  }||_o$. To do it, let $g:K \longrightarrow \mathbb{R}$ be a $\zeta$-H\"older function s.t. $H_\zeta(g), |g|_\infty \leq 1$. By Equation (\ref{poier}), we get 
	
	\begin{align*}
		\begin{split}
			\left|\int gd(\func{F}_{x\ast}\mu|_y)-\int gd(\func{F}_{y\ast}\mu|_y) \right|&=\left|\int\!{g(G(x,z))}d(\mu|_y)(z)\right.\\
			&\qquad\left.-\int\!{g(G(y,z))}d(\mu|_y)(z) \right|
		\end{split}
		\\&\leq\int{\left|G(x,z)-G(y,z)\right|}d(\mu|_y)(z)
		\\&\leq|G|_\zeta d_1(x,y)^\zeta \int{1}d(\mu|_y)(z) 
		\\&\leq|G|_\zeta d_1(x,y)^\zeta ||\mu|_y||_o.
	\end{align*}Thus, taking the supremum over $g$ and the essential supremum over $y$, we get 
	
	\begin{equation*}
		||\func{F}%
		_{x  \ast }\mu |_{y  } - \func{F}%
		_{y \ast }\mu |_{y  }||_o \leq|G|_\zeta d_1(x,y)^\zeta ||\mu||_\infty.\qedhere
	\end{equation*}
	
\end{proof}

\begin{remark}
	From now on, we denote the operator $\overline{\func {F}}_{\Phi,h}$ just by $\overline{\func {F}}_\Phi$. 
\end{remark}

For the next proposition and henceforth, for a given path $\Gamma _\mu ^\omega \in \Gamma_{ \mu }$ (associated with the disintegration $\omega = (\{\mu _\gamma\}_\gamma, \phi _1)$, of $\mu$), unless written otherwise, we consider the particular path $\Gamma_{\func{\overline{F}}_\Phi\mu} ^\omega \in \Gamma_{\func{\overline{F}}_\Phi \mu}$ defined by the Corollaries \ref{fff} and \ref{disintnorm}, by the expression

\begin{equation*}
	\Gamma_{\func{\overline{F}}_\Phi \mu} ^\omega (\gamma)=\dfrac{1}{\lambda h(\gamma)}\sum_{i=1}^{\deg(f)}{\func{F}%
		_{\gamma _i \ast }\Gamma _\mu ^\omega (\gamma_i)h(\gamma_i)e^{\varphi(\gamma_i)}}\ \ m\mathnormal{-a.e.}\ \ \gamma \in M.  \label{niceformulaaareer}
\end{equation*}
Recall that $\Gamma_{\mu} ^\omega (\gamma) = \mu|_\gamma:= \pi_{2*}(\phi_{1}(\gamma)\mu _\gamma)$ and in particular $\Gamma_{\func{\overline{F}}_\Phi\mu} ^\omega (\gamma) = (\func{\overline{F}}_\Phi\mu)|_\gamma = \pi_{2*}(\mathcal{\overline{L}}_\varphi\phi_1(\gamma)(\func{F}_\Phi\mu )_\gamma)$, where $\phi_1 = \dfrac{d \pi _{1*} \mu}{dm}$.

\begin{proposition}\label{iuaswdas}
Assume that $F\colon \Sigma \to \Sigma$ satisfies assumptions
\emph{(f1)}, \emph{(f2)}, \emph{(f3)}, \emph{(H1)}, and \emph{(H2)}, and that
\[
(\alpha \cdot L)^\zeta < 1.
\]
Then there exist constants $0<\beta<1$ and $D>0$ such that, for every
$\mu \in \mathcal{H}_\zeta^{+}$ satisfying $\phi_1=1$ $m$-almost everywhere,
and for every representative $\Gamma_\mu^\omega \in \Gamma_\mu$, the following estimate holds:
\[
\bigl|\Gamma^\omega_{\overline{F}_\Phi\mu}\bigr|_\zeta
\;\le\;
\beta\,\bigl|\Gamma_\mu^\omega\bigr|_\zeta
\;+\;
D\,\|\mu\|_\infty,
\]
where $\beta := (\alpha \cdot L)^\zeta$ and
$D := L^\zeta\bigl(\varepsilon_\varphi + |G|_\zeta\bigr).$
\end{proposition}

\begin{proof}  
	
Throughout the proof, we denote leaves by $x$ and $y$ instead of $\gamma$.
We assume that, for each $i=1,\dots,\deg(f)$, the preimages $x_i$ and $y_i$
belong to the same atom $P_i$ of the partition fixed in assumption~(H2). This holds for $m$-almost every $x,y\in M$. Moreover, by Remark~\ref{yturhfvb}, we have
\[
\frac{1}{\lambda h(x)}
\sum_{i=1}^{\deg(f)} h(x_i)e^{\varphi(x_i)}
=\overline{\mathcal{L}}_{\varphi,h}(1)(x)=1,
\quad m\text{-a.e. }x\in M.
\]We write
\begin{align*}
(\overline{F}_\Phi\mu)|_x-(\overline{F}_\Phi\mu)|_y
&=
\frac{1}{\lambda h(x)}
\sum_{i=1}^{\deg(f)} F_{x_i*}(\mu|_{x_i})\,h(x_i)e^{\varphi(x_i)}\\
&\quad-
\frac{1}{\lambda h(y)}
\sum_{i=1}^{\deg(f)} F_{y_i*}(\mu|_{y_i})\,h(y_i)e^{\varphi(y_i)}.
\end{align*}Adding and subtracting intermediate terms, we obtain
\begin{align*}
(\overline{F}_\Phi\mu)|_x-(\overline{F}_\Phi\mu)|_y
&=
\frac{1}{\lambda h(x)}
\sum_{i=1}^{\deg(f)} F_{x_i*}(\mu|_{x_i})\,h(x_i)e^{\varphi(x_i)}\\
&\quad-
\frac{1}{\lambda h(x)}
\sum_{i=1}^{\deg(f)} F_{x_i*}(\mu|_{x_i})\,h(x_i)e^{\varphi(y_i)}\\
&\quad+
\frac{1}{\lambda h(x)}
\sum_{i=1}^{\deg(f)} F_{x_i*}(\mu|_{x_i})\,h(x_i)e^{\varphi(y_i)}\\
&\quad-
\frac{1}{\lambda h(y)}
\sum_{i=1}^{\deg(f)} F_{y_i*}(\mu|_{y_i})\,h(y_i)e^{\varphi(y_i)}.
\end{align*}By Lemma~\ref{niceformulaac}, this yields
	\begin{equation}\label{skdjghdjfjkhd}
		||(\func{\overline{F}}_\Phi\mu)|_{x} - (\func{\overline{F}}_\Phi\mu)|_{y}||_o \leq \func{I_1} + \func{I_2},
	\end{equation}where
	\begin{equation*}
    \label{utrp}
		\func{I_1}:= \sum_{i=1}^{\deg(f)}{||\func{F}%
			_{x _i \ast }\mu |_{x _i}||_o \left| \dfrac{h(x_i)e^{\varphi(x_i)}}{\lambda h(x)} - \dfrac{h(y_i)e^{\varphi(y_i)}}{\lambda h(y)}\right|}
	\end{equation*}
    and
	\begin{equation*}
    \label{puipu}
		\func{I_2}:= \sum_{i=1}^{\deg(f)}{||\func{F}%
			_{x _i \ast }\mu |_{x _i } - \func{F}%
			_{y _i \ast }\mu |_{y _i }||_o\dfrac{h(y_i)e^{\varphi(y_i)}}{\lambda h(y)}}.
	\end{equation*}We first estimate $I_1$. We decompose it as
\[
\func{I}_1\le \func{I}_{1,1}+\func{I}_{1,2},
\]
where
\begin{align*}
\func{I}_{1,1}&=
\sum_{i=1}^{\deg(f)}
\|F_{x_i*}(\mu|_{x_i})\|_o
\left|
\frac{h(x_i)}{\lambda h(x)}
\bigl(e^{\varphi(x_i)}-e^{\varphi(y_i)}\bigr)
\right|,\\
\func{I}_{1,2}&=
\sum_{i=1}^{\deg(f)}
\|F_{x_i*}(\mu|_{x_i})\|_o
\frac{e^{\varphi(y_i)}}{\lambda}
\left|
\frac{h(x_i)}{h(x)}-\frac{h(y_i)}{h(y)}
\right|.
\end{align*}Using Equations (\ref{iirotyirty}), (\ref{f32}) and by Lemma \ref{niceformulaac}, we obtain
    
	\begin{eqnarray*}
		\func{I}_{1,1} &=& \sum_{i=1}^{\deg(f)}{||\func{F}%
			_{x _i \ast }\mu |_{x _i}||_o \left| \dfrac{h(x_i)e^{\varphi(x_i)}}{\lambda h(x)} - \dfrac{h(x_i)e^{\varphi(y_i)}}{\lambda h(x)}\right|}
			\\&\leq &\sum_{i=1}^{\deg(f)}{ ||\mu |_{x _i}||_o |e^{\varphi(x_i)}- e^{\varphi(y_i)}| \dfrac{h(x_i)}{\lambda h(x)} }
			\\&\leq &||\mu ||_{\infty}\sum_{i=1}^{\deg(f)}{ |e^{\varphi(x_i)}- e^{\varphi(y_i)}| \dfrac{h(x_i)}{\lambda h(x)} }
			\\&\leq &||\mu ||_{\infty}\sum_{i=1}^{\deg(f)}{ H_\zeta(e^{\varphi})d(x_i,y_i)^\zeta \dfrac{h(x_i)}{\lambda h(x)} }
			\\&\leq &||\mu ||_{\infty}\sum_{i=1}^{\deg(f)}{ \epsilon _ \varphi\inf {e^{\varphi}}L^\zeta d(x,y)^\zeta \dfrac{h(x_i)}{\lambda h(x)} }
			\\&\leq &||\mu ||_{\infty}\epsilon _ \varphi L^\zeta d(x,y)^\zeta \sum_{i=1}^{\deg(f)}{ e^{\varphi(x_i)} \dfrac{h(x_i)}{\lambda h(x)} }
			\\&=&||\mu ||_{\infty}\epsilon _ \varphi L^\zeta d(x,y)^\zeta.
	\end{eqnarray*}Since $1/h$ is $\zeta$-Hölder, there exists a constant $A_1>0$ such that
\[
\left|
\frac{h(x_i)}{h(x)}-\frac{h(y_i)}{h(y)}
\right|
\le
A_1 L^\zeta d(x,y)^\zeta.
\]Moreover, Lemma \ref{niceformulaac} yields that
	
	\begin{eqnarray*}
		\func{I}_{1,2} &=& \sum_{i=1}^{\deg(f)}{||\func{F}
			_{x_i \ast }\mu |_{x _i}||_o \left|\dfrac{h(x_i)e^{\varphi(y_i)}}{\lambda h(x)} - \dfrac{h(y_i)e^{\varphi(y_i)}}{\lambda h(y)}\right|}
			\\&\leq & \sum_{i=1}^{\deg(f)}{||\mu |_{x_i}||_o \dfrac{e^{\varphi(y_i)}}{\lambda} \left| \dfrac{h(x_i)}{h(x)} - \dfrac{h(y_i)}{h(y)}\right|}
			\\&\leq & ||\mu ||_\infty L^\zeta d(x,y)^\zeta A_1 \sum_{i=1}^{\deg(f)}{ \dfrac{e^{\varphi(y_i)}}{\lambda}}
			\\&\leq & ||\mu ||_\infty L^\zeta d(x,y)^\zeta A_1 \sum_{i=1}^{\deg(f)}{ \dfrac{h(y)h(y_i) e^{\varphi(y_i)}}{h(y)h(y_i)\lambda} }
			\\&\leq & \dfrac{\sup h}{\inf h} ||\mu ||_\infty L^\zeta  d(x,y)^\zeta A_1 \sum_{i=1}^{\deg(f)}{\dfrac{h(y_i) e^{\varphi(y_i)}}{h(y)\lambda}}
			\\&\leq & \dfrac{\sup h}{\inf h} ||\mu ||_\infty L^\zeta  d(x,y)^\zeta A_1.
	\end{eqnarray*}Combining the above estimates yields
\begin{equation}\label{glhkgf}
\func{I}_1
\le
\|\mu\|_\infty L^\zeta d(x,y)^\zeta
\left(
\varepsilon_\varphi+
A_1\frac{\sup h}{\inf h}
\right).
\end{equation}We now estimate $I_2$. Since $x_i,y_i\in P_i$, Lemma~\ref{apppoas} gives
	\begin{eqnarray*}
		\func{I_2}&=& \sum_{i=1}^{\deg(f)}{||\func{F}%
			_{x _i \ast }\mu |_{x _i } - \func{F}%
			_{y _i \ast }\mu |_{y _i }||_o\dfrac{h(y_i)e^{\varphi(y_i)}}{\lambda h(y)}}
		\\&\leq& \sum_{i=1}^{\deg(f)}{\dfrac{h(y_i)e^{\varphi(y_i)}}{\lambda h(y)}\left(\alpha ^\zeta |\mu|_\zeta d_1(x_i, y_i)^\zeta + |G|_\zeta d_1(x_i, y_i)^\zeta ||\mu |_{y_i}||_o\right)}
		\\&\leq& \left(\alpha ^\zeta |\mu|_\zeta L^\zeta d_1(x, y)^\zeta  + |G|_\zeta L^\zeta d_1(x, y)^\zeta ||\mu ||_\infty \right)\sum_{i=1}^{\deg(f)}{\dfrac{h(y_i)e^{\varphi(y_i)}}{\lambda h(y)} }
		\\&=& \alpha ^\zeta |\mu|_\zeta L^\zeta d_1(x, y)^\zeta  + |G|_\zeta L^\zeta d_1(x, y)^\zeta ||\mu ||_\infty. 
	\end{eqnarray*}Hence, 
	\begin{equation}\label{esti2}
		\func{I_2} \leq \alpha ^\zeta |\mu|_\zeta L^\zeta d_1(x, y)^\zeta  + |G|_\zeta L^\zeta d_1(x, y)^\zeta ||\mu ||_\infty.
	\end{equation}Finally, combining
\eqref{skdjghdjfjkhd}, \eqref{glhkgf}, and \eqref{esti2}, we obtain
\[
|\Gamma_{\overline{F}_\Phi\mu}^\omega|_\zeta
\le
(\alpha \cdot L)^\zeta|\Gamma_\mu^\omega|_\zeta
+
L^\zeta
\left(
\varepsilon_\varphi+
|G|_\zeta+
A_1\frac{\sup h}{\inf h}
\right)
\|\mu\|_\infty,
\]
which concludes the proof.
\end{proof}

By iterating the inequality
\[
|\Gamma_{\func{\overline{F}}_\Phi \mu}^\omega|_{\zeta}
\leq
\beta\, |\Gamma_{\mu}^\omega|_\zeta + D \|\mu\|_\infty,
\]
established in Proposition~\ref{iuaswdas}, and performing a standard inductive argument, we obtain the following consequence. Since the proof is routine, it is omitted.

\begin{corollary}\label{kjdfhkkhfdjfh}
	Suppose that $F:\Sigma \longrightarrow \Sigma$ satisfies conditions \emph{(f1)}, \emph{(f2)}, \emph{(f3)}, \emph{(H1)}, and \emph{(H2)}, and that $(\alpha L)^\zeta<1$. Then, for every $\mu \in \mathcal{H}_\zeta^{+}$ such that $\phi_1=1$ $m$-a.e., one has
	\begin{equation}\label{erkjwr}
		|\Gamma_{\func{\overline{F}}_\Phi^{\,n}\mu}^\omega|_{\zeta}
		\leq
		\beta^n |\Gamma_\mu^\omega|_\zeta
		+
		\frac{D}{1-\beta}\,\|\mu\|_\infty,
	\end{equation}
	for all $n\geq 1$, where the constants $\beta$ and $D$ are as in Proposition~\ref{iuaswdas}.
\end{corollary}

\begin{remark}\label{kjedhkfjhksjdf}
	Taking the infimum over all paths $\Gamma_{ \mu } ^\omega  \in \Gamma_{ \mu }$ and all $\Gamma_{\func{\overline{F}}_\Phi^n\mu}^\omega  \in \Gamma_{\func{\overline{F}}_\Phi^n\mu}$ on both sides of inequality (\ref{erkjwr}), we get 
	
	\begin{equation}\label{fljghlfjdgkdg}
		|\func{\overline{F}}_\Phi ^n\mu|_{\zeta}  \leq \beta^n |\mu|_\zeta + \dfrac{D}{1-\beta}||\mu||_\infty. 
	\end{equation}The above Equation (\ref{fljghlfjdgkdg}) will give a uniform bound (see the proof of Theorem \ref{riirorpdf}) for the H\"older's constant of the measure $\func{\overline{F}}_\Phi^{n} m_1$, for all $n$. Where $m_1$ is defined as the product $m_1=m \times m_2$, for a fixed probability measure $m_2$ on $K$. The uniform bound will be useful later on (see Theorem \ref{disisisii}).
	
\end{remark}

\begin{remark}\label{riirorpdf}
	Consider the probability measure $m_1$ defined in Remark \ref{kjedhkfjhksjdf}, i.e., $m_1=m \times m_2$, where $m_2$ is a given probability measure on $K$ and $\nu$ is the conformal measure fixed in the subsection \ref{hf}. Besides that, consider its trivial disintegration $\omega_0 =(\{m_{1 \gamma}  \}_{\gamma}, \phi_1)$, given by $m_{1 \gamma} = \func{\pi _{2,\gamma}^{-1}{_*}}m_2$, for all $\gamma$ and $\phi _1 \equiv 1$. According to this definition, it holds that 
	\begin{equation*}
		m_1|_\gamma = m_2, \ \ \forall \ \gamma.
	\end{equation*}In other words, the path $\Gamma ^{\omega _0}_{m_1}$ is constant: $\Gamma ^{\omega _0}_{m_1} (\gamma)= m_2$ for all $\gamma$. Hence, $m_1 \in \mathcal{H} _\zeta^{+}$.  Moreover, for each $n \in \mathbb{N}$, let $\omega_n$ be the particular disintegration of the measure $\func{\overline{F}}_\Phi^nm_1$ defined from $\omega_0$ as an application of Corollary \ref{disintnorm} and, by a simple induction, consider the path $\Gamma^{\omega_{n}}_{\func{\overline{F}}_\Phi^n m_1}$ associated with $\omega _n$. This path will be used in the proof of the next proposition.

\end{remark}

For the next result, we recall that, by Theorem \ref{belongsss}, $F$ has a unique invariant measure $\mu _{0}\in \mathbf{S}^{\infty }$. We will prove that $\mu_0$ has a regular disintegration in a way that $\mu _0 \in \mathcal{H}_\zeta^+$ (similar results are presented in \cite{GLu} and \cite{LiLu}, for others sort of systems). This will be used to prove exponential decay of correlations over the set of $\zeta$-H\"older functions.

\begin{proof}[Proof of Theorem \ref{regg}]

	Consider the path $\Gamma^{\omega_n}_{\func{\overline{F}}_\Phi^n}m_1$, defined in Remark \ref{riirorpdf},  which represents the measure $\func{\overline{F}}_\Phi^nm_1$.

	
	According to Theorem \ref{belongsss}, let $\mu _{0}\in \mathbf{S}^\infty$ be the unique $F$-invariant probability measure in $\mathbf{S}^\infty$. Consider the measure $m_1$, defined in Remark \ref{riirorpdf} and its iterates $\func{\overline{F}}_\Phi^n(m_1)$. By Theorem \ref{spgapp}, these iterates converge to $\mu _{0}$
	in $\mathbf{L}^{\infty }$. It implies that the sequence $\{\Gamma_{\func{F{_\ast }}^n(m_1)} ^{\omega _n}\}_{n}$ converges $m_1$-a.e. to $\Gamma_{\mu _{0}}^\omega\in \Gamma_{\mu_0 }$ (in $\mathcal{SB}(K)$ with respect to the metric defined in Definition \ref{wasserstein}), where $\Gamma_{\mu _{0}}^\omega$ is a path given by the Rokhlin Disintegration
	Theorem and $\{\Gamma_{\func{\overline{F}}_\Phi^n(m_1)} ^{\omega_n}\}_{n}$ is defined in Remark (\ref{riirorpdf}). It implies that $\{\Gamma_{\func{\overline{F}}_\Phi^n(m_1)} ^{\omega_n}\}_{n}$ converges pointwise to $\Gamma_{\mu _{0}}^\omega$ on a full measure set $\widehat{M}\subset M$. Let us denote $%
	\Gamma_{n}:=\Gamma^{\omega_n}_{\func{\overline{F}}_\Phi^n(m_1)}|_{%
		\widehat{M}}$ and $\Gamma:=\Gamma^\omega _{\mu _{0}}|_{\widehat{M}}$. Since $\{\Gamma_{n} \}_n $ converges pointwise to $\Gamma$, it holds $|\Gamma_{n}|_\zeta \longrightarrow |\Gamma|_\zeta$ as $n \rightarrow \infty$. Indeed, let $x,y \in \widehat{M}$. Then,
	
	\begin{eqnarray*}
		\lim _{n \longrightarrow \infty} {\dfrac{||\Gamma_n (x) - \Gamma _n(y)||_W}{d_1(x,y)^\zeta}} &= & \dfrac{||\Gamma (x) - \Gamma (y)||_W}{d_1(x,y)^\zeta}.
	\end{eqnarray*} On the other hand, by Corollary \ref{kjdfhkkhfdjfh}, the argument of the left hand side is bounded by $|\Gamma_n|_\zeta \leq  \dfrac{D}{1-\beta}$  for all $n\geq 1$. Then, 
	\begin{eqnarray*}
		\dfrac{||\Gamma (x) - \Gamma (y)||_W}{d_1(x,y)^\zeta}&\leq & \dfrac{D}{1-\beta}.
	\end{eqnarray*} Thus, $|\Gamma^\omega_{\mu _0}|_\zeta \leq\dfrac{D}{1-\beta}$ and taking the infimum we get $|\mu _0|_\zeta \leq\dfrac{D}{1-\beta}$.

\end{proof}

\subsection{Exponential Decay of Correlations}\label{jhdfjhdf}

In Section~\ref{dfjgsghdfjasdf}, we established exponential decay of correlations for observables 
$g_1 \in B'_i$ ($i \in \{1, \infty\}$) and $g_2 \in \Theta_{\mu_0}^i$ ($i \in \{1, \infty\}$) 
(see Theorems~\ref{slkdgjsdg} and \ref{shkjfjdhsf}). 
We now investigate these sets further, showing that they contain the space of H\"older functions. 
These results follow from the regularity of the disintegration of the $F$-invariant measures proved in the previous section, 
together with the properties of the norm $||\cdot||_o$.

Throughout, $\ho_\zeta(\Sigma)$ denotes the space of $\zeta$-H\"older functions on $\Sigma$, 
and $H_\zeta$ denotes the space of H\"older functions on $M$. Moreover, for any function $g$, we define
\[
|g|_\zeta := H_\zeta(g) + |g|_\infty.
\]

\begin{proof}[Proof of Theorem \ref{disisisii}]
First, we prove that $\ho_\zeta(\Sigma) \subset B'_1$ and $\ho_\zeta(\Sigma) \subset B'_\infty$. 
Suppose that $\mu \in \mathbf{S}^1$. For a given function $g \in \ho_\zeta(\Sigma)$, denote 
\[
g|_\gamma = g_\gamma \circ \pi_{2,\gamma}^{-1},
\] 
where $g_\gamma$ is the restriction of $g$ to the leaf $\gamma$. Then we have
\begin{eqnarray*}
\int g \, d\mu &=& \int_M \int_K g|_\gamma \, d\mu|_\gamma \, d\nu
\\&=& \int_M \left|\int_K g|_\gamma \, d\mu|_\gamma \right| d\nu
\\&\leq& \int_M \|\mu|_\gamma\|_o \, |g|_\zeta \, d\nu
\\&=& \|\mu\|_1 \, |g|_\zeta
\\&\leq& \|\mu\|_{S_1} \, K_3,
\end{eqnarray*}
where $K_3 := |g|_\zeta$. Hence, $\ho_\zeta(\Sigma) \subset B'_1$.  The proof that $\ho_\zeta(\Sigma) \subset B'_\infty$ follows by an analogous argument.

To prove the remained inclusions, suppose that $s \in \ho_\zeta(\Sigma)$.	Let $(\{\mu_{0, \gamma}\}_\gamma, h)$ be the disintegration of $\mu _0$ and denote by $\kappa$ the measure $\kappa:=s\mu_0$ ($ s\mu_0(E) := \int _E {s}d\mu _0$). If $(\{\kappa_{ \gamma}\}_\gamma, \widehat{\kappa} )$ is the disintegration of $\kappa$, then it holds $\widehat{\kappa} \ll \nu$ and $\kappa _\gamma \ll \mu_{0, \gamma}$ (see Lemma \ref{hdgfghddsfg}). Moreover, denoting $\overline{s}:=\dfrac{d\widehat{\kappa}}{d\nu}$ as (\ref{fjgh}) and
	\begin{equation*}
		\dfrac{d\kappa _\gamma}{d\mu_{0, \gamma}} (y) =
		\begin{cases}
			\dfrac{s(\gamma,y) }{\overline{s}(\gamma)} ,& \ \ \textnormal{if} \ \  	\overline{s}(\gamma) \neq 0 \\
			& \\
			\quad  0,&\ \ \textnormal{if} \ \  	\overline{s}(\gamma) = 0.
		\end{cases}
	\end{equation*}Moreover, denoting 
\[
\overline{s} := \frac{d\widehat{\kappa}}{d\nu} \quad \text{as in (\ref{fjgh})},
\] 
we have
\begin{equation*}
\frac{d\kappa_\gamma}{d\mu_{0,\gamma}}(y) =
\begin{cases}
\dfrac{s(\gamma,y)}{\overline{s}(\gamma)}, & \text{if } \overline{s}(\gamma) \neq 0,\\[1mm]
0, & \text{if } \overline{s}(\gamma) = 0.
\end{cases}
\end{equation*}It is immediate that $\kappa \in \mathbf{L}^\infty$, and consequently $\kappa \in \mathbf{L}^1$. 
Let us check that $\overline{s} \in H_\zeta$ by estimating its H\"older constant. 
For $\gamma_1, \gamma_2 \in M$, we have
\begin{eqnarray*}
|\overline{s}(\gamma_2) - \overline{s}(\gamma_1)| 
&=& \left|\int_{K} s(\gamma_2, y) \, d(\mu_0|_{\gamma_2}) - \int_{K} s(\gamma_1, y) \, d(\mu_0|_{\gamma_1}) \right| \\
&\leq& \left|\int_{K} s(\gamma_2, y) \, d(\mu_0|_{\gamma_2} - \mu_0|_{\gamma_1}) \right|
\\&+& \left|\int_{K} (s(\gamma_2, y) - s(\gamma_1, y)) \, d(\mu_0|_{\gamma_1}) \right| \\
&\leq& |s|_\zeta \, \|\mu_0|_{\gamma_2} - \mu_0|_{\gamma_1}\|_W
+ H_\zeta(s) \, d_1(\gamma_2, \gamma_1)^\zeta \, |\phi_1|_\infty \\
&\leq& |s|_\zeta \, |\mu_0|_\zeta \, d_1(\gamma_2, \gamma_1)^\zeta 
+ H_\zeta(s) \, |\phi_1|_\infty \, d_1(\gamma_2, \gamma_1)^\zeta.
\end{eqnarray*}

Thus, $\overline{s} \in H_\zeta$. 
Hence, $\ho_\zeta(\Sigma) \subset \Theta^\infty_{\mu_0}$ and $\ho_\zeta(\Sigma) \subset \Theta^1_{\mu_0}$.

\end{proof}

\subsection{Extending the set of potentials and proving Theorem \ref{S}}

The following proposition is taken from \cite{RA3} (more precisely, it is part
of Proposition~5.1 therein). We will use it in the proof of Theorem~\ref{S}.

We say that two potentials $\bar{\Phi}, \Phi : M \times K \to \mathbb{R}$ are
\emph{co-homologous} if there exists a continuous function
$u : M \times K \to \mathbb{R}$ such that
\[
\Phi = \bar{\Phi} - u + u \circ F.
\]

\begin{proposition}\label{homologo}
Let $\bar{\Phi} : M \times K \to \mathbb{R}$ be a H\"older continuous potential.
Then there exists a H\"older continuous potential
$\Phi : M \times K \to \mathbb{R}$, independent of the stable direction, such
that:
\begin{enumerate}
	\item $\Phi$ is co-homologous to $\bar{\Phi}$;
	\item $P(F,\bar{\Phi}) = P(F,\Phi)$;
	\item $(F,\Phi)$ and $(F,\bar{\Phi})$ have the same equilibrium states.
\end{enumerate}
\end{proposition}

\begin{proof}[Proof of Theorem~\ref{S}]
By Equation~(8) in \cite{RA3}, if $F$ is continuous, satisfies
\eqref{iurytea} for some $y_0 \in K$, and
$\bar{\Phi} : M \times K \to \mathbb{R}$ is H\"older continuous, then the
fibre-constant potential $\Phi$ defined by
\[
\Phi(x,y) := \bar{\Phi}(x,y_0), \qquad \forall (x,y) \in M \times K,
\]
is co-homologous to $\bar{\Phi}$. This is precisely the potential whose existence
is guaranteed by Proposition~\ref{homologo}. In particular, the pairs
$(F,\Phi)$ and $(F,\bar{\Phi})$ have the same equilibrium states.

Moreover, $\Phi$ can be written as
\[
\Phi = \bar{\varphi}_{y_0} \circ \pi_1,
\]
where $\bar{\varphi}_{y_0} : M \to \mathbb{R}$ is defined by
\[
\bar{\varphi}_{y_0}(x) := \bar{\Phi}(x,y_0), \qquad \forall x \in M.
\]
Therefore, if $\bar{\varphi}_{y_0} \in \mathscr{P}_M$, all conclusions of the theorem
follow, and the proof is complete.
\end{proof}

\begin{remark}
	If the fibre $K$ can be decomposed as a finitely union $K=K_1\cup \cdots \cup K_n$ of pairwise disjoint compact sets $K_1, \cdots, K_n$ then the condition $G(x, y_0)=y_0$ for all $x\in M$ in Theorem \ref{S} can be replaced by $G_i(x, y_i)=y_i$ for all $x\in M$ and some $y_i\in K_i$, $i=1, \cdots, n$.
	In fact, since $M\times K$ is a product space and $K=K_1\cup \cdots \cup K_n$ we may define $n$ fibre dynamics $G_i:M\times K_i\to K_i$ by $G_i(x, y)=G(x, y)$ when $y\in K_i$, for each $i=1,\cdots, n$. Further details can be found in \cite{RA3}. 
\end{remark}

\subsection{Examples}\label{dkjfhksjdhfksdf}

In this section, we present several examples illustrating the applicability of the results obtained in this work.  
In some cases, namely, Examples~\ref{new}, \ref{sesprowerpo}, \ref{sesprowerpoo}, \ref{uruitruytidjf}, and \ref{poi}, we analyze either the base dynamics $f$ or the fibre dynamics $G$ separately, without explicitly constructing the full skew-product $(f,G)$. By this, we mean that any combination of such base and fibre dynamics gives rise to a skew-product to which all Theorems~A through~J apply. On the other hand, Example~\ref{tsujii} considers a complete skew-product structure where Theorems~A through~J are directly applicable. Example~\ref{ferradura} corresponds to a continuous dynamics scenario with uniform contraction along all fibres while preserving a fixed horizontal fibre. Consequently, Theorem~\ref{S} applies in this setting.

\begin{example} \label{new}
	
	Let $f_0: M \longrightarrow M$ be a map defined by $f_0(x,y)=(\id(x),3y\mod1)$, where $M:=[0,1]^2$ is endowed with the $\mathbb{T}^2$ topology and $\id$ is the identity map on $[0,1]$. On $[0,1]^2$, we consider the metric $$d_1((x_0, y_0), (x_1,y_1)) = \max {d(x_0,y_0), d(x_1,y_1)},$$ where $d$ is the metric of $[0,1]$. This system has $(0,0) = (1,1)$ and all points of the horizontal segment $[0,1]\times \{1/2\}$ as fixed points.
	
	Consider the partition $P_0= [0,1/3] \times [0,1]$, $P_1= [1/3,2/3] \times [0,1]$, and $P_2= [2/3,1]\times[0,1]$. In particular, the fixed point $p_0=(1/2,1/2) \in \mathcal{A}:= \inte P_1$ (where $\inte P_1$ means the interior of $P_1$).
	
	For a given $\delta>0$, consider a perturbation $f$ of $f_0$, given by $f(x,y)=(g(x), 3y\mod1)$ such that $g(1/2)=1/2$, $0< g'(1/2) <1$ and $g$ is $\delta$-close to $\id(x)=x$ in the $C^2$ topology. Moreover, suppose that $|g'(x)| \geq k_0 > 1$ for all $x \in P_0 \cup P_2$. In particular, without loss of generality, suppose that $1 - \delta < g'(1/2) < 1+ \delta <3$ and $\deg(g)=1$. Below, the reader can see the graph of such a function $g$.

	\begin{figure}[htp]
		\begin{center}
			\includegraphics[width=0.5\textwidth]{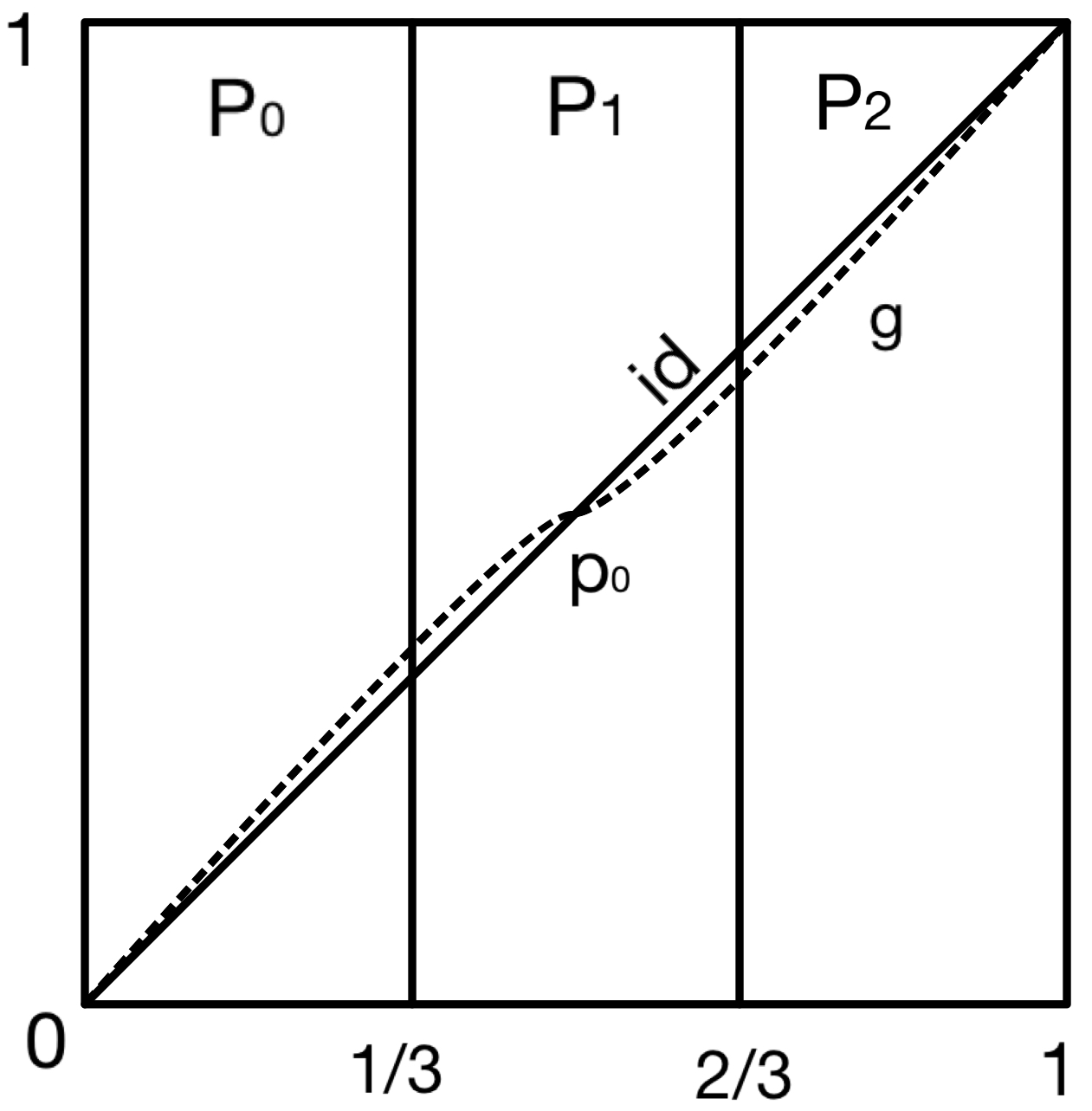}
			\caption{The graph of the perturbed map $g$}\label{AIMS}
		\end{center}
	\end{figure}Thus, we have
	$$
	Df (1/2, 1/2)=\left( \begin{array}{ll}
		g'(1/2) & 0 \\
		0 & 3
	\end{array}\right).
	$$And since $p_0=(1/2,1/2)$ is still a fixed for $f$ we have that $p_0$ becomes a saddle point (for $f$) as in the next Example \ref{sesprowerpo}. Moreover, since $\deg(g)=1$, we have that $\deg(f_0)=\deg(f)=3$, $q=1$, $\sigma = 3$, $L(x,y) := 1/ g'(x), \forall (x,y) \in \mathbb{T}^2$. In general, we have the following expression for the derivative of $f$:
	$$
	Df(x,y)= \left( \begin{array}{ll}
		g'(x) & 0 \\
		0 & 3
	\end{array}\right),
	$$ for all $(x,y) \in [0,1]^2$. This expression ensures that $\rho$ is $\zeta$-Holder for $0< \zeta \leq 1$. Therefore, for every $\epsilon>0$, there exists $\delta>0$ such that $L(x,y) \in (1-\epsilon, 1+\epsilon)$ for all $(x,y) \in [0,1]^2$, so that $L_1$ can be defined as $L_1 := 1 + \epsilon$.
	
	Since $g:[0,1] \longrightarrow [0,1]$ is $\delta$-close to $\id:[0,1] \longrightarrow [0,1]$, we have that $1-\delta < g'(x) < 1+\delta$ for all $x \in [0,1]$ and $3(1-\delta) < \det Df(x,y) < 3(1+\delta)$ for all $(x,y) \in [0,1]^2$. Thus, 
	\begin{eqnarray*}
		\sup _{(x,y)} \log \dfrac{1}{\det Df (x,y)} - \inf _{(x,y)} \log \dfrac{1}{\det Df (x,y)} &\leq& \log \dfrac{1}{3(1-\delta)} - \log \dfrac{1}{3(1+\delta)} \\&=& \log \dfrac{(1+\delta)}{(1-\delta)}.
	\end{eqnarray*}
	Therefore, it holds
	\begin{equation}\label{dfdfds}
		\sup _{(x,y) \in [0,1]^2} \log \dfrac{1}{\det Df (x,y)} - \inf _{(x,y) \in [0,1]^2} \log \dfrac{1}{\det Df (x,y)} \leq \log \dfrac{(1+\delta)}{(1-\delta)}.
	\end{equation}Note that, since $\e ^{\epsilon _\rho} \approx 1$, $L_1 \approx 1$, $0<\zeta \leq 1$ and $q(L_1^\zeta[1+(L_1-1)^\zeta]) \approx 1$ we have that 
	\begin{equation*}
		\e ^{\epsilon_\rho} \cdot \left( \dfrac{(\deg(f) - q)\sigma ^{-\zeta} + qL_1^\zeta[1+(L_1-1)^\zeta] }{\deg(f)}\right) \approx \dfrac{2(3^{-\zeta}) + 1}{3} < 1.
	\end{equation*}The above relation shows that the system satisfies (\ref{kdljfhkdjfkasd}).
	
	Now we will prove that this system satisfies (\ref{f32}) of (f3). Note that,
	\begin{equation*}
		\rho (x,y) := \dfrac{1}{|\det Df (x,y)|} = \dfrac{1}{3g'(x)}. 
	\end{equation*}Besides that,  since $g$ is $\delta$-close to $\id$ in the $C^2$ topology, we have 
	\begin{equation*}
		-\delta \leq g(x) -x \leq \delta,
	\end{equation*}
	\begin{equation*}
		1-\delta \leq g'(x) \leq1+ \delta
	\end{equation*}and
	\begin{equation*}
		-\delta \leq g''(x) \leq \delta.
	\end{equation*}In what follows, the point $x_2$ is obtained by an application of the Mean Value Theorem. Then, we have
	
	\begin{eqnarray*}
		\dfrac{\left| \dfrac{1}{\left| \det Df(x_0,y_0)\right|}- \dfrac{1}{\left| \det Df(x_1,y_1)\right|}\right|}{d_2((x_0,y_0), (x_1,y_1))^\zeta}&=&\dfrac{\left| \dfrac{1}{\left| 3g'(x_0)\right|}- \dfrac{1}{\left|3g'(x_1)\right|}\right|}{\max \{d_1(x_0,x_1), d_1(y_0,y_1)\}^\zeta}\\&\leq&\dfrac{1}{3}\dfrac{\left| g'(x_1)-g'(x_0)\right|}{d_1(x_0,x_1)^\zeta}\dfrac{1}{|g'(x_1)g'(x_0)|}\\&\leq&\dfrac{1}{3}\left| g''(x_2)\right|\dfrac{1}{|g'(x_1)g'(x_0)|}\\&\leq&\dfrac{1}{3}\delta \dfrac{1}{(1-\delta)^2}=\dfrac{1}{3}\sqrt{\delta}\sqrt{\delta} \dfrac{1}{(1-\delta)^2}\\&\leq&\dfrac{1}{3}\dfrac{1}{1+\delta}\sqrt{\delta} \dfrac{1}{(1-\delta)^2}; \ \text{for small} \ \delta \\&\leq&\dfrac{\sqrt{\delta}}{(1-\delta)^2} \inf_{x\in [0,1]} \dfrac{1}{3g'(x)}.
	\end{eqnarray*}Thus,
	
	\begin{equation}\label{jhsdgjfsa}
		H_\zeta (\rho) \leq \dfrac{\sqrt{\delta}}{(1-\delta)^2} \inf_{x\in [0,1]} \rho.
	\end{equation}If $\delta$ is small enough, by Equations (\ref{dfdfds}) and (\ref{jhsdgjfsa}), the perturbed system satisfies (\ref{f31}) and (\ref{f32}) of (f3).

	We emphasize that this example satisfies the hypotheses of both articles, \cite{VAC} and \cite{VM}. More precisely, it satisfies (H1), (H2) and (P) of \cite{VAC} and (H1), (H2) and (P) of \cite{VM}. In fact, by the Variational Principle, we have that $h(f)>0$.  
\end{example}

\begin{example}\label{sesprowerpo}
We present a general strategy to construct examples by perturbing either the identity map or an expanding map near the identity.

Let $f_0 \colon \mathbb{T}^d \to \mathbb{T}^d$ be an expanding map. 
Choose a finite covering $\mathcal{P}$ of $\mathbb{T}^d$ and an element $P_1 \in \mathcal{P}$ containing a periodic point $p$ (possibly a fixed point).
We then define a perturbation $f$ of $f_0$ inside $P_1$, using a pitchfork bifurcation, so that $p$ becomes a saddle point for $f$.

By construction, $f$ coincides with $f_0$ on $P_1^{c}$, where uniform expansion is preserved.
The perturbation can be chosen so as to satisfy condition~(f1), ensuring that $f$ is not excessively contracting inside $P_1$ and that it remains topologically mixing.

It is worth noting that a perturbation with these properties does not always exist.
However, whenever such a perturbation can be carried out, condition~(f3) is also satisfied.
In this case, the invariant measure $m$ is absolutely continuous with respect to the Lebesgue measure, which is expanding, conformal, and positive on open sets.
As a consequence, the system admits no periodic attractors.
\end{example}

\begin{example}\label{sesprowerpoo}
Assume, in the setting of the previous example, that $f_0$ is diagonalizable, with eigenvalues
\[
1 < 1+a < \lambda,
\]
associated respectively with the eigenvectors $e_1$ and $e_2$, and that $x_0$ is a fixed point of $f_0$.
Fix parameters $a,\varepsilon>0$ satisfying
$
\log\!\left(\frac{1+a}{1-a}\right)<\varepsilon$ and $$\e^{\varepsilon}\!\left(
\frac{(\deg f_0-1)(1+a)^{-\zeta}
+(1/(1-a))^{\zeta}\bigl[1+(a/(1-a))^{\zeta}\bigr]}
{\deg f_0}
\right)<1.$$Note that any smaller choice of $a>0$ still satisfies these inequalities.

Let $\mathcal U$ be a finite covering of $M$ by open domains of injectivity for $f_0$.
By refining the covering if necessary, we may assume that $x_0=(m_0,n_0)$ belongs to a unique element $U\in\mathcal U$.
Choose $r>0$ sufficiently small so that $B_{2r}(x_0)\subset U$.

Define $\rho=\eta_r\ast g$, where $\eta_r(z)=r^{-2}\eta(z/r)$, $\eta$ is the standard mollifier, and
\[
g(m,n)=
\begin{cases}
\lambda(1-a), & \text{if }(m,n)\in B_r(x_0),\\[2mm]
\lambda(1+a), & \text{otherwise}.
\end{cases}
\]
Finally, define a perturbation $f$ of $f_0$ by
\[
f(m,n)=\bigl(m_0+\lambda(m-m_0),\;
n_0+(\rho(m,n)/\lambda)(n-n_0)\bigr).
\]

With this construction, $x_0$ is a saddle point of $f$.
Moreover, the required assumptions are satisfied with
\[
\mathcal A=B_{2r}(x_0), \qquad L_1=\frac{1}{1-a}, \qquad \sigma=1+2a.
\]
The only non-trivial condition to verify is~(f3).

To this end, observe that
\[
\rho(x)-\rho(y)
=\int_{S}\frac{2a}{\lambda(1-a^2)}\eta_r(z)\,dz
-\int_{S'}\frac{2a}{\lambda(1-a^2)}\eta_r(z)\,dz,
\]
where
\[
S=\{z\in\mathbb{R}^2:\ x-z\in B_r(x_0),\ y-z\notin B_r(x_0)\},
\]
and
\[
S'=\{z\in\mathbb{R}^2:\ y-z\in B_r(x_0),\ x-z\notin B_r(x_0)\}.
\]
Let $x,y\in\mathbb{R}^2$ and write $|x-y|=qr$.
Define the annulus
\[
A_q=\{z\in\mathbb{R}^2:\ 1-q<|z|<1\}.
\]
Then
\[
\frac{|\rho(x)-\rho(y)|}{|x-y|^{\zeta}}
\le
\frac{2a\,\eta_r(S)}{\lambda(1-a^2)\,q^{\zeta}r^{\zeta}}
\le
\frac{2a\,\eta(A_q)/q^{\zeta}}{\lambda(1-a^2)}.
\]
Since
\[
N:=\sup_{q>0}\frac{\eta(A_q)}{q^{\zeta}}<\infty,
\]
we may choose $a>0$ sufficiently small so that
\[
\frac{2aN}{1-a}<\varepsilon.
\]
Consequently, $H_{\zeta}(\rho)<\varepsilon\,\inf\rho$, and condition~(f3) follows.
\end{example}

\begin{example}[Manneville--Pomeau map]\label{uruitruytidjf}
Let $\alpha\in(0,1)$ and consider the $C^{1+\alpha}$ local diffeomorphism
$f_\alpha\colon[0,1]\to[0,1]$ defined by
\[
f_{\alpha}(x)=
\begin{cases}
x\bigl(1+2^{\alpha}x^{\alpha}\bigr), & \text{if } 0\le x\le \tfrac12,\\[2mm]
2x-1, & \text{if } \tfrac12<x\le1.
\end{cases}
\]

Conditions~(f1) and~(f2) are clearly satisfied.
Moreover, by explicitly computing the two inverse branches of $f_\alpha$, one checks that $L=1$.

To verify condition~(f3), consider the family of potentials
$\{\varphi_{\alpha,t}\}_{t\in(-t_0,t_0)}$, defined by
\[
\varphi_{\alpha,t}=-t\log|Df_\alpha|,
\]
for $t_0>0$ sufficiently small.
This defines a family of $C^\alpha$ potentials such that
\[
\{\varphi_{\alpha,t}\}_{t\in(-t_0,t_0)}\subset\mathscr{P}_M
\quad\text{for each }\alpha\in(0,1),
\]
as we now verify.

Indeed, for any $x,y\in[0,1]$,
\begin{eqnarray*}
|\varphi_{\alpha,t}(x)-\varphi_{\alpha,t}(y)|
&=& |t|\bigl|\log|Df_\alpha(x)|-\log|Df_\alpha(y)|\bigr|\\
&=& |t|\left|\log\frac{|Df_\alpha(x)|}{|Df_\alpha(y)|}\right|.
\end{eqnarray*}
Since $|Df_\alpha|$ is bounded away from zero and infinity on $(0,1]$ and has at most polynomial growth near $x=0$, we obtain the uniform bound
\[
\left|\log\frac{|Df_\alpha(x)|}{|Df_\alpha(y)|}\right|
\le \log(2+\alpha),
\]
which yields
\[
|\varphi_{\alpha,t}(x)-\varphi_{\alpha,t}(y)|
\le |t|\log(2+\alpha).
\]
Consequently, for $|t|$ small enough, the Hölder constant of $\varphi_{\alpha,t}$
satisfies the required bound in~(f3), completing the verification.
\end{example}

\begin{example}[Discontinuous maps]\label{poi}
Fix real numbers $\alpha_1$ and $\alpha_2$ such that
\[
0 \le \alpha_1 < \alpha_2 < 1.
\]
Define $G\colon[0,1]\times[0,1]\to[0,1]$ by
\[
G(x,y)=
\begin{cases}
\alpha_1 y, & \text{if } 0 \le x \le \tfrac12,\\[2mm]
\alpha_2 y, & \text{if } \tfrac12 < x \le 1.
\end{cases}
\]

Clearly, $G$ is discontinuous along the set $\{\tfrac12\}\times[0,1]$.
Nevertheless, $G$ satisfies condition~(H2), since for every $\zeta>0$ we have
\[
|G|_\zeta=0
\]
(see Equation~\eqref{jdhfjdh}), as $G$ does not depend on the base variable inside each atom of the partition. Moreover, $G$ is uniformly contracting along the fibres, with contraction rate $\alpha_3=\max\{\alpha_1,\alpha_2\}.$
\end{example}

\begin{example}{(Fat solenoidal attractors)}\label{tsujii}
	Consider the class of dynamical systems defined by $$F:S^1 \times \mathbb{R} \longrightarrow S^1 \times \mathbb{R} \ \ F(x,y)=(lx, \alpha y + o(x)),$$where $l \geq 2$ is an integer, $0< \alpha < 1$ is a real number, and $o$ is a $C^2$ function on the unit circle $S^1$. In \cite{Tsujii}, M. Tsujii proved the existence of an ergodic probability measure $\mu$ on $S^1 \times \mathbb{R}$ such that Lebesgue almost every point is generic. That is, $$\lim _{n \rightarrow \infty } \dfrac{1}{n} \sum _{i=0}^{n-1} \delta _{F^i(x)} = \mu \ \ \textnormal{weakly.}$$This measure $\mu$ is called the SRB measure for $F$. In the same work, Tsujii characterizes the regularity of $\mu$ with respect to Lebesgue measure depending on the parameters that define $F$.

Fix an integer $l\geq 2$. Define $\mathcal{D} \subset (0,1) \times C^2(S^1, \mathbb{R})$ as the set of pairs $(\alpha, o)$ for which the SRB measure $\mu$ is absolutely continuous with respect to the Lebesgue measure on $S^1 \times \mathbb{R}$. Let $\mathcal{D}^o \subset \mathcal{D}$ denote the interior of $\mathcal{D}$taken with respect to the product topology given by the standard topology on $(0,1)$ and the $C^2$-topology on $C^2(S^1,\mathbb{R})$. The following result is Theorem 1 on page 1012 of \cite{Tsujii}.

\begin{theorem}
	Let $l^{-1}< \lambda < 1$. There exists a finite collection of $C^\infty$ functions $u_i: S^1 \longrightarrow \mathbb{R}$, $i=1, \cdots m$, such that for any $C^2$ function $g \in C^2(S^1, \mathbb{R})$, the subset of $\mathbb{R}^m$ $$\left \{ (t_1, t_ 2, \cdots, t_m) \in \mathbb{R}^m | \left(\alpha, g(x) + \sum _{i=1}^{m} t_ iu_i (x)\right) \notin \mathcal {D}^o\right\}$$is a null set with respect to the Lebesgue measure on $\mathbb{R}^m$.
\end{theorem}

We now apply our results to show not only that we can construct equilibrium states possessing all the statistical and analytical properties described in the applications section, but also that these measures coincide with the SRB measure studied by Tsujii in \cite{Tsujii}.

Let us consider the setting: $M=S^1$ (unit circle), $K=[-2,2]$, $f=lx$, $G(x,y)=\alpha y + o(x)$, $\varphi = -\log |f'|$ and $\Phi = \varphi \circ \pi _1$. Since $f'$ is constant, it follows that $\Phi \in \mathscr{P}_\Sigma$ and $\varphi \in \mathscr{P}_M$, according to definitions \ref{P} and \ref{PH}. Therefore, there exists a unique equilibrium state $\mu_0 \in \mathbf{S}^1 \cap \mathbf{S}^\infty$ for the pair $(F, \Phi)$. Moreover, all Theorems A through J apply to $(F,\mu_0)$. Additionally, since $\varphi = -\log |f'|$, the associated conformal measure $\nu$ coincides with Lebesgue measure $m$ on $S^1$. This implies $\pi_{1*}\mu_0 \ll m$, and by Theorem A of \cite{VAC}, $\pi_{1*}\mu_0$ is ergodic.

We claim that $\mu_0$ coincides with the physical measures, $\mu$, constructed in \cite{Tsujii}. Indeed, since $\mu$ is ergodic and absolutely continuous with respect to the Lebesgue measure on $S^1 \times K$, its projection $\pi_{1*}\mu$ is also ergodic and absolutely continuous with respect to $m$ on $\mathbf{S}^1$. By uniqueness, we must have $\pi_{1*}\mu = \pi_{1*}\mu_0$, implying that both $\mu$ and $\mu_0$ project onto the same measure. By Proposition \ref{kjdhkskjfkjskdjf}, this implies $\mu = \mu_0$. The same argument applies even when $\mu$ is singular with respect to Lebesgue measure, provided that $\mu$ is physical.

\end{example}

The next example was introduced in \cite{DHRS}, where it was shown that the non-wandering set of $F$ is partially hyperbolic for certain fixed parameters satisfying the conditions stated below. This class of maps was subsequently studied in \cite{LOR}, \cite{RA2}, \cite{RA1}, and \cite{RA3}. In \cite{RA2}, the inverse map $F^{-1}$ was considered and shown to admit a skew product structure, where the base dynamics is strongly topologically mixing and non-uniformly expanding, while the fibre dynamics is uniformly contracting. In \cite{RA3}, the author established the existence of equilibrium states and proved stability results for this class of systems.

\begin{example}{(Partially hyperbolic horseshoes)} \label{ferradura} Consider the cube $R= [0,1]\times[0,1]\times[0,1]\subset\mathbb{R}^3$  and   the parallelepipeds
		$$R_0 =[0,1]\times [0,1]\times [0,1/6]\quad \mbox{and} \quad R _1=[0,1]\times [0,1]\times [5/6,1].$$ 
		Consider a map   defined for $(x,y,z)\in R_0$ as  
		$$ F_{0}(x,y,z) =(\rho x , f(y),\beta z),$$
		where $0 < \rho <{1/3}$, $\beta> 6$ and  $$f(y) =\frac {1}{1 - \left(1-\frac{1}{y}\right)e^{-1}}.$$
		Consider also a map  defined for $(x,y,z)\in R_1$ as
		$$F_{1}(x,y,z)  = \left(\frac{3}{4}- \rho x , \sigma (1 - y) ,\beta_{1} \left(z - \frac{5}{6} \right)\right),$$
		where  $0<\sigma< {1/3}$ and $3< \beta_1 < 4$.
		Define the horseshoe map $F$ on $R$ as
		$$F\vert_{R_0}=F_0,\quad F\vert_{R_1}=F_1,
		$$
		with $R\setminus(R_0\cup R_1)$ being mapped injectively outside $R$. In \cite{DHRS} it was proved that the non-wandering set of $F$ is partially hyperbolic when we consider fixed parameters satisfying conditions above.  
	\end{example}

\section{Statements and Declarations}

\subsection{Competing Interests/Conflict of interest}

This work was partially supported by CNPq (Brazil) Grants 409198/2021-8, CNPq (Brazil) Grants 446515/2024-8, CNPq, (Brazil) Grants 420353/2025-9 and Alagoas Research Foundation-FAPEAL (Brazil) Grants E:60030.0000000161/2022.

\subsection{Data availability statement}

On behalf of all authors, the corresponding author states that data sharing is not applicable to this article as no datasets were generated or analysed during the current study.

\end{document}